\documentclass[11pt]{amsart}
\usepackage[margin=1.1in]{geometry}
\usepackage{amsmath,amssymb,amsthm,mathtools}
\usepackage{graphicx,booktabs,enumitem,array}
\usepackage{tikz}
\usepackage[hidelinks]{hyperref}
\usepackage{xurl}
\hypersetup{pdftitle={Uniqueness of four-body convex central configurations for all masses},
  pdfauthor={Tejasvi Singh Tomar}}

\newtheorem{theorem}{Theorem}[section]
\newtheorem{lemma}[theorem]{Lemma}
\newtheorem{proposition}[theorem]{Proposition}
\newtheorem{corollary}[theorem]{Corollary}

\theoremstyle{definition}

\theoremstyle{remark}
\newtheorem{remark}[theorem]{Remark}
\theoremstyle{plain}
\newtheorem*{thmA}{Theorem A}
\newtheorem*{corB}{Corollary B}
\newtheorem*{corC}{Corollary C}
\numberwithin{equation}{section}
\newcommand{\R}{\mathbb{R}}
\newcommand{\C}{\mathbb{C}}
\newcommand{\dr}{\mathrm{d}r}
\DeclareMathOperator{\tr}{tr}

\newcommand{\lmax}{\lambda_{\max}}
\newcommand{\cS}{\mathcal{S}}
\newcommand{\cZ}{\mathcal{Z}}
\newcommand{\cT}{\mathcal{T}}
\newcommand{\ip}[2]{\langle #1,#2\rangle}

\begin{document}

\title[Uniqueness of four-body convex central configurations for all masses]{Uniqueness of four-body convex central\\ configurations for all masses}
\author{Tejasvi Singh Tomar}
\renewcommand{\urladdrname}{{\itshape ORCID}}
\urladdr{\href{https://orcid.org/0000-0003-4668-3639}{https://orcid.org/0000-0003-4668-3639}}
\subjclass[2020]{70F10, 70F15, 58K05}
\keywords{Central configurations, four-body problem, relative equilibria, uniqueness, nondegeneracy, Morse index}

\begin{abstract}
We prove that for every choice of four positive masses and every cyclic ordering of the bodies there is exactly one strictly convex planar central configuration with that ordering, up to similarity. This proves a conjecture of Albouy, Fu and Sun, also called the Simó--Yoccoz conjecture, and answers Problem~10 in the list of Albouy, Cabral and Santos. Uniqueness follows by a covering argument from the nondegeneracy of these configurations. We prove nondegeneracy by hand: Dziobek's relations reduce it to an eigenvalue bound for a $2\times2$ matrix that depends only on the shape of the configuration. As consequences, every degenerate planar four-body central configuration is concave, and the known symmetry theorems for kites, isosceles trapezoids and rhombi follow in a few lines.
\end{abstract}

\maketitle
\enlargethispage{6pt}

\section{Introduction}\label{sec:intro}

\subsection{Central configurations}
Consider four bodies with masses $m_1,\dots,m_4>0$ and positions $q_1,\dots,q_4\in\R^2$. Write $q=(q_1,\dots,q_4)\in(\R^2)^4\cong\R^8$ and $r_{ij}=|q_i-q_j|$, and set
\[
U(q)=\sum_{i<j}\frac{m_im_j}{r_{ij}},\qquad M=\sum_i m_i,\qquad c=\frac1M\sum_i m_iq_i,\qquad I(q)=\sum_i m_i|q_i-c|^2 .
\]
A collision-free configuration $q$ is a \emph{central configuration} (CC) if
\begin{equation}\label{eq:cc}
\nabla_{q_i}U(q)+\lambda\,m_i(q_i-c)=0,\qquad i=1,\dots,4,
\end{equation}
for some $\lambda\in\R$. Pairing \eqref{eq:cc} with $q-c$ and using that $U$ is homogeneous of degree $-1$ and translation invariant shows that necessarily $\lambda=U(q)/I(q)>0$.
Central configurations are the configurations that admit homographic motions. In the plane, each CC generates a relative equilibrium of the Newtonian four-body problem, that is, a solution that rotates rigidly about the centre of mass. The set of CCs is invariant under the similarities of the plane (translations, rotations, dilations and reflections), and CCs are counted modulo similarity. General references are \cite{Moe15,Sma70}.

A planar four-body CC is either collinear or has no three bodies on a line (Lemma~\ref{lem:nocollinear}). In the second case it is either \emph{convex}, meaning that the bodies are the vertices of a strictly convex quadrilateral, or \emph{concave}, meaning that one body lies in the interior of the triangle formed by the other three. A convex configuration determines the cyclic order of the bodies along its boundary, up to reversal. Four labelled bodies have three such cyclic orders, represented by $(1234)$, $(1243)$ and $(1324)$.

MacMillan and Bartky \cite{MB32} proved that for every choice of masses and each cyclic order there is a convex CC with that cyclic order. Xia \cite{Xia04} gave a simpler proof, and a variational proof is given in \cite{Moe15}. Albouy, Fu and Sun conjectured that this convex CC is unique for every choice of positive masses \cite{AFS08}. Albouy, Cabral and Santos included the question in their list of open problems on the classical $n$-body problem, asking whether, for the Newtonian potential, the convex CC with a given cyclic order is unique up to similarity \cite[Problem~10]{ACS12}. Following a suggestion of Albouy, Santoprete \cite{San26} calls the expected positive answer the Simó--Yoccoz conjecture. The number of four-body CCs as a function of the masses was explored numerically by Simó \cite{Sim78}.

\subsection{Previous results}
Uniqueness was known in several cases. For four equal masses, Albouy \cite{Alb95,Alb96} proved that every CC has an axis of symmetry and classified them all, so that the only convex one is the square. When the masses at the ends of a diagonal are equal, the convex CC is symmetric with respect to the other diagonal \cite{LS02,PCS07,AFS08}, and uniqueness then follows from Leandro's analysis of kite configurations \cite{Lea03}. Roberts \cite{Rob25} gave a new proof of the uniqueness of convex kites, based on the Poincaré--Hopf theorem. For two pairs of equal masses at adjacent vertices, Fernandes, Llibre and Mello \cite{FLM17} proved that the isosceles trapezoid is the unique convex CC; see also \cite{Xie12}. Santoprete proved uniqueness among co-circular CCs \cite{San21a} and among trapezoidal CCs \cite{San21b}, two classes also studied in \cite{CR12,CCLP19}, and Corbera, Cors and Roberts showed that a convex CC with perpendicular diagonals is a kite \cite{CCR18}. Uniqueness also holds when three of the masses are sufficiently small \cite{CCLM15}, and Sun, Xie and You \cite{SXY23} proved it by interval arithmetic for $m_2=1$ the largest mass and $(m_1,m_3,m_4)\in[0.6,1]^3$.

Corbera, Cors and Roberts \cite{CCR19} described the set of normalized convex CCs as a graph over an explicit three-dimensional domain. They expressed the hope that the normalized mass map on this domain is injective, which would prove uniqueness. Santoprete \cite{San26} decomposed the constrained Hessian at a planar CC into a positive semidefinite part and a signed part of rank two, and pointed out that if all convex CCs with a given ordering were nondegenerate minima, Morse-theoretic arguments could yield uniqueness. That approach had been carried out within the co-circular and trapezoidal classes \cite{San21a,San21b}, and for kites in \cite{Rob25}. Sun, Xie and You \cite[Theorem~4.1]{SXY25} proved by interval arithmetic that rhombus CCs are nondegenerate, and Santoprete \cite[Proposition~7.1]{San26} showed that their Morse index is $0$.

Concave CCs exist for all masses \cite{Ham02}, but they are not unique. There are degenerate concave kites and fold bifurcations, and the number of concave kites for given masses can be $0$, $1$ or $2$ \cite{Rob25,LX25,LPRX26}. Hampton and Moeckel proved that the number of four-body CCs is finite \cite{HM06}; see \cite{AK12} for five bodies and \cite{Moe01} for Dziobek configurations.

\subsection{Results}
Nondegeneracy is meant as in Section~\ref{sec:shape}: the similarity class is a nondegenerate critical point of $UI^{1/2}$ on the shape space, as in \cite{Moe15,Pac87}.

\begin{thmA}
Let $m_1,\dots,m_4>0$. For each cyclic order of the four bodies there is exactly one convex central configuration with that cyclic order, up to similarity. It is a nondegenerate local minimum of $UI^{1/2}$ on the shape space, and it depends real-analytically on the masses.
\end{thmA}

Up to orientation-preserving similarities there are exactly two such configurations, and they are mirror images of each other.

For a CC $q$ let
\[
Q(v)=D^2\Bigl(U+\frac{\lambda}{2}I\Bigr)(q)[v,v],\qquad v=(v_1,\dots,v_4)\in(\R^2)^4,
\]
where $\lambda=U(q)/I(q)$ is held fixed, and let
\[
K(v)=\sum_{i<j}\frac{3m_im_j}{r_{ij}^3}\Bigl(\frac{\ip{q_i-q_j}{v_i-v_j}}{r_{ij}}\Bigr)^{2}.
\]
Thus $Q$ is the Hessian of the Lagrange function for the extremal problem ``$U$ restricted to a level set of $I$'', and $K\ge0$ is the part of $Q$ that comes from the second derivatives of the pair potentials in the radial directions (Lemma~\ref{lem:hessian}). By infinitesimal rigidity, $K(v)=0$ if and only if $v$ is an infinitesimal translation or rotation (Lemma~\ref{lem:rigid}).

At every convex CC we show that $Q(v)>0$ for all $v$ with $K(v)>0$, so that every convex CC is a nondegenerate local minimum of Morse index $0$ (Corollary~\ref{cor:nondeg}). The best constant $c$ in a bound $Q\ge cK$ at all convex CCs is $\frac23$. In a companion paper \cite{TomB} we prove, with computer assistance, that $Q(v)>\frac23K(v)$ for all $v$ with $K(v)>0$, and give formal proofs in Lean of this bound and of Theorem~A. In Section~\ref{sec:optimal} we show that $\frac23$ cannot be improved, using isosceles trapezoids in which the two masses at one of the parallel sides tend to zero.

\begin{corB}
Fix positive masses $m_1,\dots,m_4$.
\begin{enumerate}[label=\textup{(\roman*)},leftmargin=2.2em]
\item For each cyclic order and orientation, the convex CC, as a point of shape space, is a real-analytic function of the masses.
\item There are exactly three convex CCs up to similarity, and exactly six up to orientation-preserving similarity. All of them are nondegenerate local minima of $UI^{1/2}$. Consequently, the Morse-theoretic lower bound of $34$ planar CCs up to orientation-preserving similarity \cite{Moe15} holds for every choice of masses for which the concave CCs are nondegenerate.
\item Every degenerate planar four-body CC is concave.
\end{enumerate}
\end{corB}

Uniqueness also recovers the known symmetry theorems, whose earlier proofs rely on a detailed analysis of the Dziobek equations. Here they follow in a few lines.

\begin{corC}
Label the bodies so that the cyclic order is $(1234)$, and let $q$ be the convex CC.
\begin{enumerate}[label=\textup{(\alph*)},leftmargin=2.2em]
\item If $m_1=m_3$, then $q$ is symmetric with respect to the diagonal $q_2q_4$ \textup{\cite{AFS08}}.
\item If $m_1=m_2$ and $m_3=m_4$, then $q$ is an isosceles trapezoid with $q_1q_2\parallel q_3q_4$ \textup{\cite{FLM17}}.
\item If $m_1=m_3$ and $m_2=m_4$, then $q$ is a rhombus \textup{\cite{LS02,PCS07,AFS08}}.
\end{enumerate}
\end{corC}

\subsection{Method}
The route from nondegeneracy to uniqueness is well known \cite{San21a,San21b,Rob25,San26}. Suppose that every convex CC is a nondegenerate minimum. Then the pairs (convex shape, masses) that form CCs make up a manifold, and its projection to the simplex of masses is a local diffeomorphism. The projection is proper, because Shub's lemma excludes collisions and the bound $Q\ge0$ at convex CCs excludes collinear limits. It is therefore a covering map, and its degree is one at equal masses, where the symmetry theorem of Albouy, Fu and Sun \cite{AFS08} shows that the square is the only convex CC. The difficulty is to prove nondegeneracy for all masses at once.

Dziobek's relations $m_im_j(r_{ij}^{-3}-\lambda/M)=\sigma A_iA_j$, where the $A_l$ are oriented triangle areas and $\sigma<0$, turn the indefinite part of $Q$ into a negative multiple of a square,
\[
Q=K-|\sigma|\,|L|^2,\qquad L(v)=\sum_{l=1}^4A_lv_l\in\R^2 .
\]
We found this identity independently of \cite{San26}; Remark~\ref{rem:identity} compares it with Santoprete's decomposition. Comparing $L$ with $K$ by the Cauchy--Schwarz inequality, we find that the generalized eigenvalues of the pair $(Q,K)$ are $1,1,1,1-\mu_1,1-\mu_2$, where $\mu_1\ge\mu_2$ are the eigenvalues of a positive semidefinite $2\times2$ matrix $G$. It is built from the weighted covariance of the six edge midpoints and the matrix $\sum_lA_l\,q_lq_l^{\mathsf T}$, with weights that do not depend on the masses. Uniqueness therefore follows from the inequality $\lmax(G)<1$. Its proof bounds the trace of $G$ by weighted sums of squared distances from well-chosen centres. The sharp bound $Q>\frac23K$ is the inequality $\lmax(G)<\frac13$. Santoprete's numerical experiments had indicated that $\lmax(G)<\frac13$ (personal communication).

Section~\ref{sec:prelim} collects the preliminaries, Section~\ref{sec:identity} proves the identity and identifies $G$, Section~\ref{sec:handN} proves $\lmax(G)<1$, with the numerical facts in Appendix~\ref{app:handN}, and Section~\ref{sec:proofs} gives the covering argument. Section~\ref{sec:optimal} treats the trapezoids, and Section~\ref{sec:cons} proves the corollaries and states some open questions.

\section{Preliminaries}\label{sec:prelim}

Throughout, $m_1,\dots,m_4>0$. The six unordered pairs $e=ij$, $1\le i<j\le4$, are called \emph{edges}. For an edge $e=ij$ we write $q_e=q_i-q_j$, $v_e=v_i-v_j$, $r_e=|q_e|$, $R_e=r_e^2$ and $s_e=r_e^{-3}$. We identify $\R^2$ with $\C$, write $J$ for multiplication by $\sqrt{-1}$ (rotation by $\pi/2$), and let $J$ act diagonally on $(\R^2)^4=\C^4$. The mass inner product is $\ip{u}{v}_M=\sum_im_i\ip{u_i}{v_i}$.

\subsection{The equations and the Hessian}
Let $q$ be a CC with multiplier $\lambda=U/I$, and set
\[
\lambda'=\frac{\lambda}{M},\qquad w_e=s_e-\lambda',\qquad \omega_e=m_im_jw_e\quad(e=ij).
\]
Since $\nabla_{q_i}U=\sum_{j\ne i}m_im_js_{ij}(q_j-q_i)$ and $m_i(q_i-c)=-\frac1M\sum_{j}m_im_j(q_j-q_i)$, equation \eqref{eq:cc} is equivalent to
\begin{equation}\label{eq:equil}
\sum_{j\ne i}\omega_{ij}\,(q_j-q_i)=0,\qquad i=1,\dots,4 .
\end{equation}
Conversely, if a collision-free $q$ satisfies \eqref{eq:equil} for some $\lambda'\in\R$, then $q$ satisfies \eqref{eq:cc} with $\lambda=M\lambda'$. Pairing with $q-c$ then shows that $\lambda=U/I$, so $q$ is a CC.

Because $I=\frac1M\sum_{i<j}m_im_jr_{ij}^2$, the function $U_\lambda=U+\frac{\lambda}{2}I$, with $\lambda$ frozen at its value at $q$, is a sum of pair potentials:
\[
U_\lambda=\sum_{i<j}m_im_j\,\varphi(r_{ij}),\qquad \varphi(r)=\frac1r+\frac{\lambda'}{2}r^2 .
\]
One computes $\varphi'(r)/r=-(r^{-3}-\lambda')$ and $\varphi''(r)-\varphi'(r)/r=3r^{-3}$. For an edge $e$ and a vector $v\in(\R^2)^4$ let
\[
\dr_e(v)=\frac{\ip{q_e}{v_e}}{r_e}
\]
be the first variation of $r_e$, and put $\kappa_e=3m_im_js_e>0$ for $e=ij$.

\begin{lemma}\label{lem:hessian}
Let $q$ be a CC and $Q=D^2U_\lambda(q)$. Then:
\begin{enumerate}[label=\textup{(\alph*)},leftmargin=2em]
\item $Q(v)=K(v)-\sum_e\omega_e|v_e|^2$, where $K(v)=\sum_e\kappa_e\,\dr_e(v)^2$;
\item $Q(q,v)=3\lambda\ip{q-c}{v}_M$ for all $v$, and in particular $Q(q,q)=3U$;
\item the three-dimensional space $\cT$ spanned by the translations and by $J(q-c)$ is contained in $\ker Q$.
\end{enumerate}
\end{lemma}

\begin{proof}
(a) For $x\in\R^2\setminus\{0\}$ the Hessian of $x\mapsto\varphi(|x|)$ is $\varphi''\,\hat x\hat x^{\mathsf T}+\frac{\varphi'}{r}(\mathrm{Id}-\hat x\hat x^{\mathsf T})=(\varphi''-\frac{\varphi'}{r})\hat x\hat x^{\mathsf T}+\frac{\varphi'}{r}\mathrm{Id}$, where $r=|x|$ and $\hat x=x/r$. Evaluating at $x=q_e$ in the direction $v_e$ and summing over the edges gives (a).
(b) Since $U$ is homogeneous of degree $-1$, $\nabla U$ is homogeneous of degree $-2$, so $D^2U(q)[q,v]=-2\ip{\nabla U(q)}{v}=2\lambda\ip{q-c}{v}_M$ by \eqref{eq:cc}. Also $D^2I(q)[u,v]=2\ip{u-\bar u}{v-\bar v}_M$, where $\bar u=\frac1M\sum_im_iu_i$, so $\frac\lambda2D^2I(q)[q,v]=\lambda\ip{q-c}{v}_M$. Adding gives $Q(q,v)=3\lambda\ip{q-c}{v}_M$, and $Q(q,q)=3\lambda I=3U$.
(c) $U_\lambda$ is invariant under translations and rotations. Differentiating the identities $U_\lambda(q+t)=U_\lambda(q)$ and $\nabla U_\lambda(e^{J\theta}q)=e^{J\theta}\nabla U_\lambda(q)$ at a critical point of $U_\lambda$ shows that translations and $Jq$ lie in $\ker Q$. Since $J(q-c)$ differs from $Jq$ by a translation, the claim follows.
\end{proof}

Moeckel \cite{Moe15} writes the same quadratic form, restricted to the normalized configuration space, as $H(x)$. In Moeckel's convention convex CCs are local minima and collinear CCs have index $n-2$.

\subsection{Shape space}\label{sec:shape}
Orientation-preserving similarities act on $\C^4$ by $q\mapsto\mu q+t\mathbf 1$, with $\mu\in\C^*$, $t\in\C$ and $\mathbf 1=(1,1,1,1)$. The quotient of the set of configurations that are not totally collapsed is the \emph{shape space}
\[
\cS=\mathbb P(\C^4/\C\mathbf 1)\cong\C P^2,
\]
a compact real four-manifold. For $i\ne j$, the normalization $q_i=0$, $q_j=1$ identifies the open set of classes with $q_i\ne q_j$ with $\C^2$, and these affine charts cover $\cS$. Let $\Delta\subset\cS$ be the collision locus, where $q_i=q_j$ for some $i\ne j$. Since $U(\mu q+t\mathbf 1)=|\mu|^{-1}U(q)$ and $I(\mu q+t\mathbf 1)=|\mu|^2I(q)$, the function
\[
f_m=U\,I^{1/2}
\]
descends to a real-analytic function on $\cS\setminus\Delta$. From $d(UI^{1/2})=I^{1/2}\bigl(dU+\frac{U}{2I}dI\bigr)$ we see that $[q]$ is a critical point of $f_m$ if and only if $q$ is a CC. We call a CC \emph{nondegenerate} if $[q]$ is a nondegenerate critical point of $f_m$, and define its \emph{Morse index} to be the index of $D^2f_m([q])$, where the \emph{index} of a quadratic form is the largest dimension of a subspace on which it is negative definite. Scaling all masses by the same factor multiplies $f_m$ by a constant, so the critical points depend only on the normalized masses.

\begin{lemma}\label{lem:shape}
Let $q$ be a CC and let $\mathcal W$ be the space of vectors $v$ with $\sum_im_iv_i=0$, $\ip{q-c}{v}_M=0$ and $\ip{J(q-c)}{v}_M=0$. Then $(\R^2)^4=\cT\oplus\R(q-c)\oplus \mathcal W$, the three summands are $Q$-orthogonal, and $Q(q-c)=3U>0$. The differential of $q\mapsto[q]$ maps $\mathcal W$ isomorphically onto $T_{[q]}\cS$, and under this isomorphism $D^2f_m([q])$ corresponds to $I^{1/2}\,Q|_\mathcal W$. Consequently the Morse index of $q$ is the index of $Q$, and $q$ is nondegenerate if and only if $\ker Q=\cT$. In particular $q$ is a nondegenerate local minimum of $f_m$ if and only if $Q(v)>0$ for every $v\notin\cT$.
\end{lemma}

\begin{proof}
The translations are $\ip{\cdot}{\cdot}_M$-orthogonal to the vectors with $\sum_im_iv_i=0$, and $q-c$, $J(q-c)$ are $\ip{\cdot}{\cdot}_M$-orthogonal to each other. This gives the direct sum decomposition. The space $\cT\oplus\R(q-c)$ is the tangent space at $q$ of the orbit of the similarity group, and $\dim \mathcal W=4=\dim\cS$, so the differential of the projection maps $\mathcal W$ isomorphically onto $T_{[q]}\cS$. The summand $\cT$ is $Q$-orthogonal to everything by Lemma~\ref{lem:hessian}(c). For $w\in \mathcal W$, Lemma~\ref{lem:hessian}(b) gives $Q(q-c,w)=Q(q,w)=3\lambda\ip{q-c}{w}_M=0$. It also gives $Q(q-c)=Q(q)=3U$. Finally, differentiating $F=UI^{1/2}$ twice at a critical point, and using $dU=-\frac{U}{2I}dI$ there, gives
\[
D^2F(q)=I^{1/2}\Bigl[Q-\frac{3U}{4I^2}\,dI\otimes dI\Bigr].
\]
Here $dI(w)=2\ip{q-c}{w}_M=0$ for $w\in \mathcal W$. Since $[q]$ is a critical point, $D^2f_m([q])$ evaluated on the image of $w$ equals $D^2F(q)[w,w]$.
\end{proof}

\subsection{Oriented areas}
For three points put $[ijk]=\frac12\det(q_j-q_i,\,q_k-q_i)$, and define
\[
A_1=[234],\qquad A_2=-[134],\qquad A_3=[124],\qquad A_4=-[123].
\]
For the square $q=(1,J,-1,-J)$, whose vertices are in counterclockwise order, the areas are $A_1=A_3=1$ and $A_2=A_4=-1$.

\begin{lemma}\label{lem:areas}
\begin{enumerate}[label=\textup{(\alph*)},leftmargin=2em]
\item $\sum_lA_l=0$ and $\sum_lA_lq_l=0$.
\item $A_l(\mu q+t\mathbf 1)=|\mu|^2A_l(q)$ for $\mu\in\C^*$, $t\in\C$, and $A_l(\bar q)=-A_l(q)$.
\item If $q_i=q_j$ with $i\ne j$, then $A_k=0$ for $k\notin\{i,j\}$. If the $q_l$ are not all equal, then all $A_l$ vanish if and only if the four points are collinear.
\item Suppose that no three of the $q_l$ are collinear. Then every $A_l$ is nonzero, and exactly one of the following holds.
\begin{itemize}
\item Two of the $A_l$ are positive and two are negative. Then $q$ is a strictly convex quadrilateral whose diagonals join the two bodies with positive $A_l$ and the two bodies with negative $A_l$. The sign pattern $(+,-,+,-)$ occurs if and only if $q_1,q_2,q_3,q_4$ are the vertices of a strictly convex quadrilateral in counterclockwise order.
\item Three of the $A_l$ have the same sign. Then $q$ is concave, and the body whose $A_l$ has the other sign lies in the interior of the triangle formed by the other three.
\end{itemize}
\end{enumerate}
\end{lemma}

\begin{proof}
(a) Consider the $4\times4$ matrix whose columns are $(1,q_l)^{\mathsf T}\in\R^3$ together with a fourth row that repeats one of its three rows. Its determinant vanishes. Expanding along the repeated row gives $\sum_l(-1)^{l+1}\,2[\,\cdot\,]_l\,(1,q_l)=0$, where $[\,\cdot\,]_l$ is the triangle formed by the bodies other than $l$. This is the statement $\sum_lA_l(1,q_l)=0$.
(b) Complex multiplication by $\mu$ is a rotation composed with a dilation by $|\mu|$, and conjugation is a reflection.
(c) The first statement is clear. If not all points coincide, choose $q_i\ne q_j$. If all $A_l$ vanish, the two triangles containing $q_i$ and $q_j$ are degenerate, so all points lie on the line through $q_i$ and $q_j$.
(d) The $A_l$ are nonzero because no triangle is degenerate. Let $S$ be the sum of the positive $A_l$. By (a), $S$ is also the sum of the absolute values of the negative $A_l$, and
\[
p=\frac1S\sum_{A_l>0}A_lq_l=\frac1S\sum_{A_l<0}|A_l|q_l .
\]
If two of the $A_l$ are positive, $p$ lies in the open segment joining the two corresponding points and also in the open segment joining the other two. Two segments that cross at interior points are the diagonals of a convex quadrilateral, which is strictly convex because no three points are collinear. If three $A_l$ are positive, $p$ is the fourth point and is a convex combination of the other three with positive weights. For the pattern $(+,-,+,-)$ the diagonals are $q_1q_3$ and $q_2q_4$, so the cyclic order is $(1234)$. Moreover $A_1=[234]>0$ means that the triangle $q_2q_3q_4$ is positively oriented, that is, the quadrilateral is traversed counterclockwise. Conversely, let $q_1q_2q_3q_4$ be a counterclockwise strictly convex quadrilateral. No vertex lies inside the triangle of the others, so two $A_l$ are positive and two are negative, and the equal-sign pairs are the diagonals $\{1,3\}$ and $\{2,4\}$. Finally $A_1=[234]>0$, because three consecutive vertices of a counterclockwise convex polygon form a positively oriented triangle.
\end{proof}

\subsection{Infinitesimal rigidity}
Let $\dr:(\R^2)^4\to\R^6$ be the linear map $v\mapsto(\dr_e(v))_e$. A vector $(\omega_e)\in\R^6$ satisfying \eqref{eq:equil} is called a \emph{self-stress} of $q$. For $\omega\in\R^6$ and $v\in(\R^2)^4$,
\begin{equation}\label{eq:stressdual}
\sum_e\omega_er_e\,\dr_e(v)=\sum_{i<j}\omega_{ij}\ip{q_i-q_j}{v_i-v_j}=-\sum_i\Bigl\langle v_i,\sum_{j\ne i}\omega_{ij}(q_j-q_i)\Bigr\rangle .
\end{equation}

\begin{lemma}\label{lem:rigid}
Suppose that no three of $q_1,\dots,q_4$ are collinear. Then $\ker\dr$ is the space $\cT$ of infinitesimal translations and rotations, $\dr$ has rank $5$, and the self-stresses of $q$ are the multiples of $(A_iA_j)_{ij}$. The image of $\dr$ is the hyperplane $\tau^\perp$, where
\[
\tau_e=A_iA_j\,r_e\qquad(e=ij).
\]
In particular, if $q$ is a CC then $K(v)=0$ if and only if $v\in\cT$.
\end{lemma}

\begin{proof}
Let $v\in\ker\dr$. The triangle $q_1q_2q_3$ is nondegenerate, hence infinitesimally rigid. After subtracting an element of $\cT$ we may therefore assume $v_1=v_2=v_3=0$. Then $\ip{q_4-q_i}{v_4}=0$ for $i=1,2$. The vectors $q_4-q_1$ and $q_4-q_2$ span $\R^2$, so $v_4=0$. Hence $\ker\dr=\cT$ and $\dr$ has rank $5$. By \eqref{eq:stressdual}, $\omega$ is a self-stress if and only if $(\omega_er_e)_e$ is orthogonal to the image of $\dr$, which is a hyperplane, so the self-stresses form a line. By Lemma~\ref{lem:areas}(a),
\[
\sum_{j\ne i}A_iA_j(q_j-q_i)=A_i\Bigl(\sum_jA_jq_j-q_i\sum_jA_j\Bigr)=0,
\]
so $(A_iA_j)$ is a self-stress, and it is nonzero. Therefore $\operatorname{im}\dr=\tau^\perp$. The last statement follows because all $\kappa_e>0$.
\end{proof}

\subsection{Dziobek's relations}

\begin{lemma}\label{lem:nocollinear}
A planar CC of four bodies is either collinear or has no three bodies on a line.
\end{lemma}

\begin{proof}
Suppose that $q_1,q_2,q_3$ lie on a line $\ell$ and $q_4\notin\ell$, and let $n$ be a unit normal to $\ell$. For $i\le3$, pairing \eqref{eq:equil} with $n$ gives $\omega_{i4}\ip{q_4-q_i}{n}=0$. Since $\ip{q_4-q_i}{n}\ne0$, we get $w_{i4}=0$, that is, $r_{14}=r_{24}=r_{34}=\lambda'^{-1/3}$. Then the three distinct points $q_1,q_2,q_3$ of $\ell$ lie on a circle centred at $q_4$, which is impossible.
\end{proof}

\begin{lemma}[Dziobek \cite{Dzi00}]\label{lem:dziobek}
Let $q$ be a noncollinear planar CC of four bodies. Then there is a constant $\sigma<0$ such that
\begin{equation}\label{eq:dziobek}
m_im_j\,w_{ij}=\sigma\,A_iA_j\qquad\text{for all }i<j .
\end{equation}
\end{lemma}

\begin{proof}
By Lemma~\ref{lem:nocollinear} no three bodies are collinear. By \eqref{eq:equil}, $(\omega_e)$ is a self-stress, so Lemma~\ref{lem:rigid} gives \eqref{eq:dziobek} for some $\sigma\in\R$. If $\sigma=0$, all six distances equal $\lambda'^{-1/3}$, which is impossible in the plane. Note that $w_e>0$ if and only if $r_e<\lambda'^{-1/3}$. Suppose $\sigma>0$. If $q$ is convex, then by Lemma~\ref{lem:areas}(d) $A_iA_j>0$ on the diagonals and $A_iA_j<0$ on the sides, so every diagonal is shorter than every side. But if the diagonals are $q_iq_k$ and $q_jq_l$ and meet at $p$, the triangle inequality gives
\[
r_{ik}+r_{jl}=|q_i-p|+|p-q_k|+|q_j-p|+|p-q_l|>r_{ij}+r_{kl},
\]
a contradiction. If $q$ is concave, say with $q_4$ inside the triangle $q_1q_2q_3$, then $A_iA_4<0<A_iA_j$ for $i,j\le3$. Every interior edge $q_iq_4$ would then be longer than every exterior edge. But writing $q_4=\alpha q_1+\beta q_2+\gamma q_3$ with $\alpha,\beta,\gamma>0$ and $\alpha+\beta+\gamma=1$, we get $r_{14}\le\beta r_{12}+\gamma r_{13}<\max(r_{12},r_{13})$, again a contradiction.
\end{proof}

\begin{proposition}\label{prop:convexsigns}
Let $q$ be a convex CC whose cyclic order is $(1234)$. Then:
\begin{enumerate}[label=\textup{(\alph*)},leftmargin=2em]
\item $w_{12},w_{23},w_{34},w_{14}>0>w_{13},w_{24}$. Equivalently,
\[
\min(r_{13},r_{24})>\lambda'^{-1/3}>\max(r_{12},r_{23},r_{34},r_{14}),
\]
so each diagonal is longer than each side.
\item $w_{12}w_{34}=w_{13}w_{24}=w_{14}w_{23}$.
\item A longest side and a shortest side are opposite each other.
\end{enumerate}
\end{proposition}

\begin{proof}
(a) follows from \eqref{eq:dziobek}, $\sigma<0$, and the signs of $A_iA_j$. For (b), multiply the relations for complementary pairs of edges: $m_1m_2m_3m_4\,w_{12}w_{34}=\sigma^2A_1A_2A_3A_4$, and likewise for the other two pairings. For (c), suppose that $q_1q_2$ is a longest side. Since $w$ is a decreasing function of $r$, we have $0<w_{12}\le w_{14},w_{23}$. Then $w_{34}=w_{14}w_{23}/w_{12}\ge\max(w_{12},w_{14},w_{23})$, so $q_3q_4$ is a shortest side.
\end{proof}

These facts are classical; see \cite[Section~2.2]{CCR19} and the references there.

\section{An identity for the Hessian and the matrix \texorpdfstring{$G$}{G}}\label{sec:identity}

\subsection{The identity}

\begin{proposition}\label{prop:identity}
Let $q$ be a noncollinear planar CC of four bodies, and let $\sigma<0$ be as in Lemma~\ref{lem:dziobek}. Then
\begin{equation}\label{eq:identity}
Q(v)=K(v)-|\sigma|\,|L(v)|^2,\qquad L(v)=\sum_{l=1}^4A_lv_l\in\R^2 ,
\end{equation}
for all $v\in(\R^2)^4$.
\end{proposition}

\begin{proof}
By Lemma~\ref{lem:hessian}(a) and \eqref{eq:dziobek}, $Q(v)=K(v)-\sigma\sum_{i<j}A_iA_j|v_i-v_j|^2$. Expanding the squares and using $\sum_jA_j=0$,
\[
\sum_{i<j}A_iA_j|v_i-v_j|^2=\frac12\sum_{i,j}A_iA_j\bigl(|v_i|^2+|v_j|^2-2\ip{v_i}{v_j}\bigr)=-\Bigl|\sum_iA_iv_i\Bigr|^2 .
\]
Hence $Q=K+\sigma|L|^2=K-|\sigma||L|^2$.
\end{proof}

Since $K\ge0$ and $L$ takes values in $\R^2$, $Q$ is positive semidefinite on the kernel of $L$, which has codimension at most two. With Lemma~\ref{lem:shape} this gives the following bound, due to Palmore \cite{PalI,PalII}; see \cite[Proposition~21]{Moe15} and \cite[Proposition~5.7]{San26} for other proofs.

\begin{corollary}\label{cor:palmore}
Every noncollinear planar CC of four bodies has Morse index at most $2$.
\end{corollary}

\begin{remark}\label{rem:identity}
(i) Santoprete \cite{San26} works in pair space and decomposes the constrained Hessian as $H_C=L_\Delta+\widetilde L$. The gap Laplacian $L_\Delta$ is positive semidefinite, and the transverse part $\widetilde L$ has rank two for noncollinear four-body CCs. Proposition~\ref{prop:identity} is the analogous statement in the coordinates of the bodies: $K$ plays the role of $L_\Delta$, and $-|\sigma||L|^2$ is an explicit square that plays the role of $\widetilde L$, which Santoprete also writes in factored form \cite[Corollary~5.3]{San26}, with coefficients given by the Dziobek areas \cite[Remark~5.5]{San26}. Santoprete characterizes definiteness by an explicit $2\times2$ matrix. Proposition~\ref{prop:majorant}(c) below gives an analogous criterion with the matrix $G$.

(ii) In the language of rigidity theory, $\omega$ is a self-stress of the framework formed by the six edges, $\sum_e\omega_e|v_e|^2$ is its stress energy, and \eqref{eq:identity} says that the stress matrix of $\omega$ is the rank-one matrix $|\sigma|AA^{\mathsf T}$. The computation uses only that the self-stress is unique up to a factor and has product form. It therefore applies verbatim to Dziobek configurations of $d+2$ bodies in $\R^d$, where it gives $Q=K+\sigma|\sum_lA_lv_l|^2$ with $\sum_lA_lv_l\in\R^d$ and the $A_l$ the coefficients of the affine dependence.

(iii) Dias \cite[Lemma~7.7]{Dia26} reads the equations of a CC as an equilibrium stress of the configuration, a positive multiple of $\omega$. For strictly convex planar CCs that satisfy a sign condition of MacMillan--Bartky type \cite[Definition~7.8]{Dia26}, which for $n=4$ is Proposition~\ref{prop:convexsigns}(a) (compare \cite[Corollary~7.10]{Dia26}), Dias shows that the stress matrix is positive semidefinite of rank $n-3$, so that these CCs are not vertically degenerate \cite[Corollary~7.12]{Dia26}. For four bodies, (ii) gives this matrix explicitly, $\Omega=|\sigma|AA^{\mathsf T}$, for convex and concave CCs alike, with stress energy $|\sigma|\,|L|^2$. Vertical degeneracy concerns higher dimensions, whereas the nondegeneracy proved here concerns the planar second variation; the rank-one structure alone does not control the sign of $Q=K-|\sigma||L|^2$, which needs the comparison with $K$ made in Section~\ref{sec:majorant}.
\end{remark}

\subsection{Expressing \texorpdfstring{$L$}{L} through the edge variations}
The form $L$ vanishes on $\cT$: $\sum_lA_l=0$ kills translations and $\sum_lA_lJq_l=J\sum_lA_lq_l=0$ kills rotations. By Lemma~\ref{lem:rigid}, $L$ therefore factors through $\dr$. The next lemma gives an explicit factorization.

\begin{lemma}\label{lem:beta}
Suppose that no three of $q_1,\dots,q_4$ are collinear. For an edge $e=ij$ and an index $l\notin e$, let $k$ be the fourth index, so that $ijk$ is the triangle of the bodies other than $l$, and put $N_{e,l}=R_{ik}+R_{jk}-R_{ij}$. Define
\[
\beta_e=-\sum_{l\notin e}\frac{r_e\,N_{e,l}}{8A_l}\,q_l\ \in\R^2 .
\]
Then $L(v)=\sum_e\beta_e\,\dr_e(v)$ and $\sum_e\tau_e\,\dr_e(v)=0$ for all $v$. Consequently $L=\sum_e(\beta_e+\tau_ey)\,\dr_e$ for every $y\in\R^2$.
\end{lemma}

\begin{proof}
The identity $\sum_lA_l(q)\,q_l=0$ of Lemma~\ref{lem:areas}(a) holds for all $q$. Differentiating it in the direction $v$ gives $L(v)=-\sum_l dA_l(v)\,q_l$. Near $q$, $A_l$ is a smooth function of the side lengths of its triangle $ijk$. Heron's formula $16A_l^2=2(R_{ij}R_{ik}+R_{ij}R_{jk}+R_{ik}R_{jk})-R_{ij}^2-R_{ik}^2-R_{jk}^2$ gives $\partial A_l/\partial r_{ij}=r_{ij}N_{ij,l}/(8A_l)$, and similarly for the other two sides, with the sign of $A_l$ taken into account. Substituting $dA_l=\sum_{e\not\ni l}(\partial A_l/\partial r_e)\,\dr_e$ gives the first identity. The second is \eqref{eq:stressdual} applied to the self-stress $(A_iA_j)$.
\end{proof}

\subsection{The majorant}\label{sec:majorant}
Let $q$ be a noncollinear planar CC. For $e=ij$ put
\begin{equation}\label{eq:De}
D_e=-\frac{w_e}{3s_eA_iA_j} .
\end{equation}
By \eqref{eq:dziobek}, $\sigma=m_im_jw_e/(A_iA_j)$, and therefore $D_e=|\sigma|/\kappa_e>0$. For $y\in\R^2$ define the positive semidefinite $2\times2$ matrix
\[
S(y)=\sum_eD_e\,(\beta_e+\tau_ey)(\beta_e+\tau_ey)^{\mathsf T} .
\]
The quantities $D_e$, $\beta_e$ and $\tau_e$ are built from the configuration and from $\lambda'$. For a convex CC, Proposition~\ref{prop:convexsigns}(b) determines $\lambda'$ from the configuration:
\begin{equation}\label{eq:lambda}
\lambda'=\frac{s_{12}s_{34}-s_{13}s_{24}}{s_{12}+s_{34}-s_{13}-s_{24}},
\end{equation}
where the denominator is positive by Proposition~\ref{prop:convexsigns}(a). On convex CCs, therefore, $S(y)$ is a function of the configuration alone, and the masses do not enter. Under $q\mapsto\mu q$ one has $D_e\mapsto|\mu|^{-4}D_e$, $\tau_e\mapsto|\mu|^5\tau_e$ and $\beta_e\mapsto|\mu|^2\frac{\mu}{|\mu|}\beta_e$. Hence the matrix $S(|\mu|^{-3}\frac{\mu}{|\mu|}y)$ of the configuration $\mu q$ is the matrix $S(y)$ of $q$ conjugated by a rotation. In particular the trace is unchanged.

\begin{proposition}\label{prop:majorant}
Let $q$ be a noncollinear planar CC of four bodies.
\begin{enumerate}[label=\textup{(\alph*)},leftmargin=2em]
\item For all $v\in(\R^2)^4$ and $y\in\R^2$,
\[
|\sigma|\,|L(v)|^2\le\lmax(S(y))\,K(v)\le\tr S(y)\,K(v),\qquad\text{hence}\qquad Q\ge\bigl(1-\tr S(y)\bigr)K .
\]
\item Let $c_0=\sum_eD_e\tau_e^2>0$, $b_0=\sum_eD_e\tau_e\beta_e\in\R^2$, $y^*=-b_0/c_0$ and $G=S(y^*)=\sum_eD_e\beta_e\beta_e^{\mathsf T}-c_0^{-1}b_0b_0^{\mathsf T}$. Then
\[
S(y)=G+c_0\,(y-y^*)(y-y^*)^{\mathsf T},
\]
so $G\preceq S(y)$ for every $y$, and $\min_y\tr S(y)=\tr G$.
\item Let $\mu_1\ge\mu_2\ge0$ be the eigenvalues of $G$. The generalized eigenvalues of the pair $(Q,K)$ on the five-dimensional space $(\R^2)^4/\cT$ are $1,1,1,1-\mu_1,1-\mu_2$. In particular $Q\ge(1-\mu_1)K$ and the constant is optimal. The Morse index of $q$ is the number of $\mu_i>1$, and $q$ is degenerate if and only if $\mu_1=1$ or $\mu_2=1$.
\end{enumerate}
\end{proposition}

\begin{proof}
(a) Let $u\in\R^2$ be a unit vector. By Lemma~\ref{lem:beta} and the Cauchy--Schwarz inequality,
\[
\ip{u}{L(v)}^2=\Bigl(\sum_e\frac{\ip{u}{\beta_e+\tau_ey}}{\sqrt{\kappa_e}}\cdot\sqrt{\kappa_e}\,\dr_e(v)\Bigr)^2\le\Bigl(\sum_e\frac{\ip{u}{\beta_e+\tau_ey}^2}{\kappa_e}\Bigr)K(v).
\]
Multiplying by $|\sigma|$ and using $D_e=|\sigma|/\kappa_e$ gives $|\sigma|\ip{u}{L(v)}^2\le u^{\mathsf T}S(y)u\,K(v)$. Take the supremum over $u$, use $\lmax\le\tr$ for positive semidefinite matrices, and apply Proposition~\ref{prop:identity}.

(b) This is completing the square. We have $c_0>0$ because $\tau\ne0$.

(c) By Lemma~\ref{lem:rigid}, $\dr$ identifies $V=(\R^2)^4/\cT$ with the hyperplane $\tau^\perp\subset\R^6$. Under this identification $K$ becomes the restriction of $k(z)=\sum_e\kappa_ez_e^2$, and by Lemma~\ref{lem:beta} $L$ becomes $z\mapsto\sum_e\beta_ez_e$. For $g\in\R^6$,
\[
\sup_{0\ne z\in\tau^\perp}\frac{(\sum_eg_ez_e)^2}{k(z)}=\min_{t\in\R}\sum_e\frac{(g_e+t\tau_e)^2}{\kappa_e} .
\]
Indeed, $\sum_eg_ez_e=\sum_e(g_e+t\tau_e)z_e$ on $\tau^\perp$, so the Cauchy--Schwarz inequality gives ``$\le$''. Equality holds for $z_e=(g_e+t^*\tau_e)/\kappa_e$, where $t^*$ is the minimizer, because this $z$ lies in $\tau^\perp$. Taking $g_e=\ip{u}{\beta_e}$ and multiplying by $|\sigma|$ gives
\[
\sup_{v\notin\cT}\frac{|\sigma|\ip{u}{L(v)}^2}{K(v)}=\min_t\sum_eD_e\bigl(\ip{u}{\beta_e}+t\tau_e\bigr)^2=u^{\mathsf T}Gu\qquad(u\in\R^2).
\]
Let $\mathcal K:V\to V^*$ be the isomorphism defined by $K$. The last identity says $|\sigma|\,L\mathcal K^{-1}L^{\mathsf T}=G$. The generalized eigenvalues of $(Q,K)$ are the eigenvalues of $\mathcal K^{-1}Q=\mathrm{Id}-|\sigma|\mathcal K^{-1}L^{\mathsf T}L$. This operator is self-adjoint for $K$. The operator $|\sigma|\mathcal K^{-1}L^{\mathsf T}L$ on the five-dimensional space $V$ has rank at most two, and its characteristic polynomial is $t^3$ times that of $|\sigma|L\mathcal K^{-1}L^{\mathsf T}=G$. This gives the list of generalized eigenvalues. The remaining statements follow from Lemma~\ref{lem:shape}, because $Q$ vanishes on $\cT$.
\end{proof}

\begin{corollary}\label{cor:criterion}
Let $q$ be a noncollinear planar CC of four bodies. If $\tr S(y)<1$ for some $y\in\R^2$, then $Q\ge(1-\tr S(y))K$, and $q$ is a nondegenerate local minimum of $f_m$.
\end{corollary}

\begin{proof}
Combine Proposition~\ref{prop:majorant}(a), Lemma~\ref{lem:rigid} and Lemma~\ref{lem:shape}.
\end{proof}

\subsection{The midpoint form of \texorpdfstring{$G$}{G}}\label{sec:midpoint}
We now identify $G$ in closed form. The matrix that appears is the weighted covariance of the six edge midpoints. For an edge $e=ij$ put
\[
n_e=q_i+q_j,\qquad d_e=D_e\tau_e^2=-\frac{w_e\,A_iA_j\,r_e^5}{3},
\]
so that $n_e$ is twice the midpoint of the edge. Let
\[
M_2=\sum_lA_l\,q_lq_l^{\mathsf T},\qquad W=\sum_ed_e,\qquad \bar n=\frac1W\sum_ed_en_e,\qquad C=\sum_ed_e\,(n_e-\bar n)(n_e-\bar n)^{\mathsf T}.
\]
At a noncollinear planar CC every $d_e$ is positive, because $D_e>0$ and $\tau_e\ne0$. Note that $W$ is the constant $c_0$ of Proposition~\ref{prop:majorant}(b). The matrix $M_2$ does not change under translations, because $\sum_lA_l=0$ and $\sum_lA_lq_l=0$.

\begin{lemma}\label{lem:M2det}
For every configuration, $\det M_2=-4A_1A_2A_3A_4$. In particular $M_2$ is invertible when no three of the $q_l$ are collinear.
\end{lemma}

\begin{proof}
Both sides are polynomials in the coordinates of $q$, so it suffices to check the identity on the open set of strictly convex configurations in the cyclic order $(1234)$. Translate the crossing of the diagonals to the origin and write
\[
q_1=-\alpha u,\qquad q_3=\gamma u,\qquad q_2=-\beta v,\qquad q_4=\delta v,
\]
with $\alpha,\beta,\gamma,\delta>0$ and unit vectors $u$, $v$ at angle $\varphi$. The relations $\sum_lA_l=0$ and $\sum_lA_lq_l=0$ give
\[
(A_1,A_3,A_2,A_4)=k\Bigl(\gamma,\ \alpha,\ -\frac{\delta(\alpha+\gamma)}{\beta+\delta},\ -\frac{\beta(\alpha+\gamma)}{\beta+\delta}\Bigr),
\]
and comparing with the area of the triangle $q_2q_3q_4$ gives $k=\frac12(\beta+\delta)\sin\varphi$. Hence
\[
M_2=k(\alpha+\gamma)\bigl(\alpha\gamma\,uu^{\mathsf T}-\beta\delta\,vv^{\mathsf T}\bigr),\qquad \det M_2=-k^2(\alpha+\gamma)^2\alpha\beta\gamma\delta\sin^2\varphi .
\]
Since $A_1A_2A_3A_4=k^4\alpha\beta\gamma\delta(\alpha+\gamma)^2/(\beta+\delta)^2$, both sides equal $-\frac14\alpha\beta\gamma\delta(\alpha+\gamma)^2(\beta+\delta)^2\sin^4\varphi$. The last statement follows from Lemma~\ref{lem:areas}(d).
\end{proof}

\begin{lemma}\label{lem:midpoint-ids}
For every configuration $q$ and every $v\in(\R^2)^4$,
\[
\sum_e\tau_e\,\dr_e(v)=0\qquad\text{and}\qquad\sum_e\tau_e\,\dr_e(v)\,n_e=-M_2\,L(v).
\]
\end{lemma}

\begin{proof}
We have $\tau_e\dr_e(v)=A_iA_j\ip{q_i-q_j}{v_i-v_j}$ for $e=ij$. Write each sum over edges as $\frac12\sum_{i\ne j}$ and expand $\ip{q_i-q_j}{v_i-v_j}$ into four terms. We use $\sum_{j\ne i}A_j=-A_i$ and $\sum_{j\ne i}A_jq_j=-A_iq_i$, from Lemma~\ref{lem:areas}(a). In the second sum, the terms with $\ip{q_i}{v_i}$ and with $\ip{q_j}{v_j}$ contribute $-2\sum_iA_i^2\ip{q_i}{v_i}\,q_i$. The two cross terms contribute $-\sum_jA_j\bigl(M_2v_j-2A_j\ip{q_j}{v_j}\,q_j\bigr)$, because
\[
\sum_{i\ne j}A_i\ip{q_i}{v_j}\,q_i=(M_2-A_jq_jq_j^{\mathsf T})\,v_j,\qquad \sum_{i\ne j}A_i\ip{q_i}{v_j}\,q_j=-A_j\ip{q_j}{v_j}\,q_j .
\]
The total is $-M_2\sum_jA_jv_j=-M_2L(v)$. The first identity is the same computation without the factor $q_i+q_j$.
\end{proof}

\begin{lemma}\label{lem:midpoint-family}
Suppose that no three of $q_1,\dots,q_4$ are collinear. For every $z\in\R^2$,
\[
L=\sum_e\tau_eM_2^{-1}(z-n_e)\,\dr_e .
\]
Conversely, if $L=\sum_eg_e\,\dr_e$ with $g_e\in\R^2$, then $g_e=\tau_eM_2^{-1}(z-n_e)$ for exactly one $z\in\R^2$.
\end{lemma}

\begin{proof}
By Lemmas~\ref{lem:midpoint-ids} and \ref{lem:M2det},
\[
\sum_e\tau_eM_2^{-1}(z-n_e)\,\dr_e(v)=M_2^{-1}\Bigl(z\sum_e\tau_e\dr_e(v)-\sum_e\tau_e\dr_e(v)\,n_e\Bigr)=L(v).
\]
If two families $(g_e)$ and $(g'_e)$ both represent $L$, then each coordinate of $(g_e-g'_e)_e$ annihilates the image of $\dr$. By Lemma~\ref{lem:rigid} that image is $\tau^\perp$, so $g_e-g'_e=\tau_ec$ for some $c\in\R^2$. Finally $\tau_eM_2^{-1}(z-n_e)+\tau_ec=\tau_eM_2^{-1}(z+M_2c-n_e)$, and $z$ is unique because $\tau\ne0$.
\end{proof}

\begin{proposition}\label{prop:midG}
Let $q$ be a noncollinear planar CC of four bodies. Then $G=M_2^{-1}\,C\,M_2^{-1}$.
\end{proposition}

\begin{proof}
By Lemma~\ref{lem:midpoint-family}, applied to the factorization of Lemma~\ref{lem:beta}, there is $z_0\in\R^2$ with $\beta_e=\tau_eM_2^{-1}(z_0-n_e)$ for every edge. Hence $\beta_e+\tau_ey=\tau_eM_2^{-1}(z-n_e)$ with $z=z_0+M_2y$, and
\[
S(y)=M_2^{-1}\Bigl(\sum_ed_e\,(z-n_e)(z-n_e)^{\mathsf T}\Bigr)M_2^{-1}=M_2^{-1}\bigl(C+W\,(z-\bar n)(z-\bar n)^{\mathsf T}\bigr)M_2^{-1}.
\]
As $y$ runs over $\R^2$, so does $z$. The minimum in Proposition~\ref{prop:majorant}(b) is attained at $z=\bar n$, and there $S(y^*)=M_2^{-1}CM_2^{-1}$.
\end{proof}

\begin{corollary}\label{cor:midpoint-test}
Let $q$ be a noncollinear planar CC and $c>0$. Then $\mu_1<c$ if and only if the $2\times2$ matrix $cM_2^2-C$ is positive definite, that is, if and only if
\[
\det\bigl(cM_2^2-C\bigr)>0\qquad\text{and}\qquad\tr\bigl(cM_2^2-C\bigr)>0 .
\]
The eigenvalues of $G$ do not depend on the labelling of the bodies and do not change under similarities.
\end{corollary}

\begin{proof}
By Proposition~\ref{prop:midG}, $cI-G=M_2^{-1}(cM_2^2-C)M_2^{-1}$, which is congruent to $cM_2^2-C$. A symmetric $2\times2$ matrix is positive definite if and only if its determinant and trace are positive. For the last statement, a relabelling permutes the $A_l$ and changes all their signs at most, so it changes $M_2$ at most by a sign and leaves every $d_e$, and hence $C$, unchanged. A similarity $q\mapsto\mu q+t\mathbf 1$ multiplies $M_2$ by $|\mu|^4$ and $C$ by $|\mu|^8$, and conjugates both by the same rotation. A reflection $R$ changes the sign of every $A_l$ and fixes every $d_e$, so $M_2\mapsto-RM_2R^{\mathsf T}$, $C\mapsto RCR^{\mathsf T}$ and $G\mapsto RGR^{\mathsf T}$, and again the eigenvalues are unchanged.
\end{proof}

\begin{remark}\label{rem:trace}
The sharp quantity is $\mu_1=\lmax(G)$. The trace $\tr S(y)$ satisfies $\tr S(y)\ge\tr G\ge\lmax(G)$ and is a sum of squares with explicit coefficients, which makes it convenient for computations \cite{TomB}.
\end{remark}

\section{A proof of \texorpdfstring{$\lmax(G)<1$}{lambda\_max(G)<1} without computer assistance}\label{sec:handN}

We prove that $\mu_1=\lmax(G)<1$ at every convex CC, where $G$ is the matrix of Proposition~\ref{prop:midG}. The proof uses the midpoint form of Section~\ref{sec:midpoint} and plane geometry. Its numerical content is a list (N1)--(N15) of elementary numerical facts with explicit rational constants (Appendix~\ref{app:handN}).

\begin{theorem}\label{thm:handN}
At every strictly convex central configuration of four positive masses, $\lmax(G)<1$.
\end{theorem}

Since $G\succeq0$, it suffices to show $\tr G<1$. For each point $y\in\R^2$ the trace is bounded by a weighted sum $B(y)$ of squared distances from $M_2^{-1}y$ to the points $M_2^{-1}n_e$, and this sum is affine in $\lambda'$ (Section~\ref{sec:hN-centre}). We take $y$ on the segment from $n_{13}$ to $n_{24}$, the doubled midpoints of the diagonals. At the two ends of the segment $B(y)$ is a sum over the vertices, with the factors $\csc^2$ of the vertex angles, and Jensen's inequality and Stewart's theorem bound it (Section~\ref{sec:hN-ends}). Along the segment $B(y)$ gains a term proportional to $|M_2^{-1}(n_{24}-n_{13})|^2$ (Section~\ref{sec:hN-test}). Either one end bound is below $1$, or the gain pays for both, and the latter reduces to one inequality in the angles $\alpha_1+\alpha_2$ and $\alpha_3-\alpha_4$ (Sections~\ref{sec:hN-large} and~\ref{sec:hN-hard}).

Throughout this section $q$ is a convex CC whose bodies are in counterclockwise order $(1234)$, and $\lambda'$ is its multiplier. By Proposition~\ref{prop:convexsigns}(a), every diagonal of $q$ is longer than every side, and $\lambda'>r_{13}^{-3}$, $\lambda'>r_{24}^{-3}$.

\subsection{Bounds with a fixed centre}\label{sec:hN-centre}
For $t\in\R$ and an edge $e=ij$ put
\[
d_e(t)=\tfrac13\,(t\,r_e^3-1)\,A_iA_jR_e,\qquad W(t)=\sum_ed_e(t),
\]
so that $d_e(\lambda')=d_e$ and $W(\lambda')=W$ (Section~\ref{sec:midpoint}). For $y\in\R^2$ put
\[
B(y,t)=\sum_ed_e(t)\,\bigl|M_2^{-1}(n_e-y)\bigr|^2 .
\]

\begin{lemma}\label{lem:hN-centre}
\begin{enumerate}[label=\textup{(\alph*)},leftmargin=2em]
\item $\tr G\le B(y,\lambda')$ for every $y\in\R^2$.
\item For each $y$, the map $t\mapsto B(y,t)$ is affine, and $B(y,0)=B_0:=\frac13\bigl(4-(\tr M_2)^2/\det M_2\bigr)\ge\frac43$.
\item $W(t)=\frac t3\sum_eA_iA_jr_e^5$, and for $y_0,y_1\in\R^2$ and $\theta\in\R$,
\[
B\bigl((1-\theta)y_0+\theta y_1,t\bigr)=(1-\theta)B(y_0,t)+\theta B(y_1,t)-\theta(1-\theta)\,W(t)\,\bigl|M_2^{-1}(y_1-y_0)\bigr|^2 .
\]
\item If $B(y,t_0)<\frac43$ for some $y\in\R^2$ and some $t_0\in(0,\lambda']$, then $\mu_1\le\tr G\le B(y,t_0)$.
\end{enumerate}
\end{lemma}

\begin{proof}
(a) Since $\sum_ed_e(n_e-\bar n)=0$, we have $\sum_ed_e(n_e-y)(n_e-y)^{\mathsf T}=C+W(y-\bar n)(y-\bar n)^{\mathsf T}$. Conjugating by $M_2^{-1}$, taking traces and using Proposition~\ref{prop:midG} gives $B(y,\lambda')=\tr G+W|M_2^{-1}(y-\bar n)|^2$. Here $W>0$, because every $d_e$ is positive (Section~\ref{sec:midpoint}).

(b) Each $d_e(t)$ is affine in $t$. At $t=0$ the weights are $-\frac13A_iA_jR_e$, and we claim
\begin{equation}\label{eq:hN-moments}
\sum_eA_iA_jR_e=0,\qquad\sum_eA_iA_jR_e\,n_e=0,\qquad\sum_eA_iA_jR_e\,n_en_e^{\mathsf T}=(\tr M_2)M_2-2M_2^2 .
\end{equation}
We write each sum over edges as half of the sum over ordered pairs $(i,j)$; the terms with $i=j$ vanish. We expand $R_{ij}=|q_i|^2+|q_j|^2-2\ip{q_i}{q_j}$ and use Lemma~\ref{lem:areas}(a). First, $\sum_jA_jR_{ij}=\sum_jA_j|q_j|^2=\tr M_2$ for every $i$. Since $\sum_iA_i=0$ and $\sum_iA_iq_i=0$, this gives the first two identities. For the third, write $n_en_e^{\mathsf T}=(q_iq_i^{\mathsf T}+q_jq_j^{\mathsf T})+(q_iq_j^{\mathsf T}+q_jq_i^{\mathsf T})$. The first part contributes $\sum_iA_iq_iq_i^{\mathsf T}\sum_jA_jR_{ij}=(\tr M_2)M_2$. In the second part the terms with $|q_i|^2$ or $|q_j|^2$ vanish, and $-2A_iA_j\ip{q_i}{q_j}q_iq_j^{\mathsf T}=-2A_iq_iq_i^{\mathsf T}\,A_jq_jq_j^{\mathsf T}$ sums to $-2M_2^2$.

By the first two identities, the masses $-\frac13A_iA_jR_e$ at the points $M_2^{-1}n_e$ have total mass $0$ and first moment $0$. So their weighted sum of squared distances to $M_2^{-1}y$ does not depend on $y$, and by the third identity it equals
\[
-\tfrac13\tr\bigl(M_2^{-1}\bigl((\tr M_2)M_2-2M_2^2\bigr)M_2^{-1}\bigr)=-\tfrac13\bigl((\tr M_2)\,\tr M_2^{-1}-4\bigr)=B_0,
\]
because $\tr M_2^{-1}=\tr M_2/\det M_2$ for $2\times2$ matrices. By Lemma~\ref{lem:M2det}, $\det M_2=-4A_1A_2A_3A_4<0$, since two of the $A_l$ are positive and two are negative. Hence $B_0\ge\frac43$.

(c) The formula for $W(t)$ follows from the first identity in \eqref{eq:hN-moments}. The function $y\mapsto B(y,t)$ is a polynomial of degree two whose quadratic part is $W(t)|M_2^{-1}y|^2$. For every such polynomial $F$, the value at $(1-\theta)y_0+\theta y_1$ equals $(1-\theta)F(y_0)+\theta F(y_1)-\theta(1-\theta)W(t)|M_2^{-1}(y_1-y_0)|^2$.

(d) By (b), $B(y,\cdot)$ is affine with $B(y,0)\ge\frac43>B(y,t_0)$ and $t_0>0$, so it is decreasing, and $B(y,\lambda')\le B(y,t_0)$. Now use (a) and $G\succeq0$.
\end{proof}

\subsection{The quadrilateral}\label{sec:hN-quad}
Let $O$ be the point where the diagonals cross (Figure~\ref{fig:quad}). Put $r_{lO}=|q_l-O|$ for $l=1,\dots,4$, and
\[
\psi=\frac{r_{1O}r_{3O}}{r_{2O}r_{4O}} .
\]
Let $\gamma\in(0,\pi)$ be the angle $\angle q_1Oq_2$, and let $\alpha_l$ be the interior angle of the quadrilateral at $q_l$. Put $\bar\alpha_3=\pi-\alpha_3$, $\bar\alpha_4=\pi-\alpha_4$, and
\[
a_{13}=\tfrac12(\cot\alpha_1-\cot\alpha_3),\qquad a_{24}=\tfrac12(\cot\alpha_2-\cot\alpha_4).
\]

\begin{figure}[ht]
\centering
\begin{tikzpicture}[scale=3]
\coordinate (O) at (0,0);
\coordinate (q1) at (-12:1.619);
\coordinate (q2) at (62.92:1.700);
\coordinate (q3) at (168:0.631);
\coordinate (q4) at (242.92:0.437);
\draw[thick] (q1) -- (q2) -- (q3) -- (q4) -- cycle;
\draw (q1) -- (q3) (q2) -- (q4);
\fill (O) circle[radius=0.018];
\node at (308:0.13) {$O$};
\node[right] at (q1) {$q_1$};
\node[above] at (q2) {$q_2$};
\node[left] at (q3) {$q_3$};
\node[below] at (q4) {$q_4$};
\path (O) -- (q1) node[midway,below] {$r_{1O}$};
\path (O) -- (q2) node[midway,left] {$r_{2O}$};
\path (O) -- (q3) node[pos=0.42,above] {$r_{3O}$};
\path (O) -- (q4) node[pos=0.6,right] {$r_{4O}$};
\draw (-12:0.22) arc[start angle=-12,end angle=62.92,radius=0.22];
\node at (25.5:0.35) {$\gamma$};
\draw (q1) ++(113.64:0.175) arc[start angle=113.64,end angle=181.69,radius=0.175];
\path (q1) ++(141:0.31) node {$\alpha_1$};
\draw (q2) ++(224.81:0.175) arc[start angle=224.81,end angle=293.64,radius=0.175];
\path (q2) ++(268.3:0.31) node {$\alpha_2$};
\draw (q3) ++(-51.19:0.13) arc[start angle=-51.19,end angle=44.81,radius=0.13];
\path (q3) ++(-31.6:0.235) node {$\alpha_3$};
\draw (q4) ++(1.69:0.13) arc[start angle=1.69,end angle=128.81,radius=0.13];
\path (q4) ++(95.9:0.225) node {$\alpha_4$};
\end{tikzpicture}
\caption{The quadrilateral of Section~\ref{sec:hN-quad}: the diagonals meet at $O$, $r_{lO}=|q_l-O|$, $\gamma$ is the angle $\angle q_1Oq_2$, and $\alpha_l$ is the interior angle at $q_l$.}\label{fig:quad}
\end{figure}

\begin{lemma}\label{lem:hN-geom}
\begin{enumerate}[label=\textup{(\alph*)},leftmargin=2em]
\item Every $\alpha_l$ exceeds $60^\circ$.
\item $|A_1|=\frac12r_{24}r_{3O}\sin\gamma$, $|A_3|=\frac12r_{24}r_{1O}\sin\gamma$, $|A_2|=\frac12r_{13}r_{4O}\sin\gamma$ and $|A_4|=\frac12r_{13}r_{2O}\sin\gamma$. Moreover $\sin\alpha_1=2|A_3|/(r_{12}r_{14})$, $\sin\alpha_2=2|A_4|/(r_{12}r_{23})$, $\sin\alpha_3=2|A_1|/(r_{23}r_{34})$ and $\sin\alpha_4=2|A_2|/(r_{14}r_{34})$.
\item \textup{(Stewart)} $r_{3O}R_{12}+r_{1O}R_{23}=r_{13}(r_{2O}^2+r_{1O}r_{3O})$, $r_{3O}R_{14}+r_{1O}R_{34}=r_{13}(r_{4O}^2+r_{1O}r_{3O})$, $r_{4O}R_{12}+r_{2O}R_{14}=r_{24}(r_{1O}^2+r_{2O}r_{4O})$ and $r_{4O}R_{23}+r_{2O}R_{34}=r_{24}(r_{3O}^2+r_{2O}r_{4O})$. Moreover $r_{2O}^2+r_{1O}r_{3O}$ and $r_{4O}^2+r_{1O}r_{3O}$ are smaller than $R_{24}$, $r_{1O}^2+r_{2O}r_{4O}$ and $r_{3O}^2+r_{2O}r_{4O}$ are smaller than $R_{13}$, and $\frac13<\psi<3$.
\item We have
\[
\cot\alpha_1+\cot\alpha_3=\frac{r_{13}(1-1/\psi)}{r_{24}\sin\gamma},\qquad \cot\alpha_2+\cot\alpha_4=-\frac{r_{24}(\psi-1)}{r_{13}\sin\gamma},
\]
\begin{align*}
2r_{24}\sin\gamma\,a_{13}&=\bigl(1+\tfrac1\psi\bigr)(r_{1O}-r_{3O})-2\cos\gamma\,(r_{2O}-r_{4O}),\\
2r_{13}\sin\gamma\,a_{24}&=(1+\psi)(r_{2O}-r_{4O})-2\cos\gamma\,(r_{1O}-r_{3O}).
\end{align*}
\item If $\alpha_l\le90^\circ$, then both sides at $q_l$ are longer than $2r\cos\alpha_l$, where $r$ is the length of the diagonal that does not end at $q_l$.
\item If $90^\circ\le\alpha_1+\alpha_2\le180^\circ$, then $1+\cos\gamma>2\sin^2(\alpha_1+\alpha_2)$.
\end{enumerate}
\end{lemma}

\begin{proof}
(a) The vertex $q_l$ and its two neighbours form a triangle whose side opposite $q_l$ is a diagonal. This side is longer than the other two, so $\alpha_l$ is the largest angle of the triangle, and $\alpha_l>60^\circ$.

(b) The triangle $q_2q_3q_4$ has area $|A_1|$, base $q_2q_4$ of length $r_{24}$ and height $r_{3O}\sin\gamma$. The other three areas are found in the same way. The triangle formed by $q_l$ and its neighbours is the triangle of the bodies other than the vertex opposite to $q_l$, and twice its area is the product of the two sides at $q_l$ and $\sin\alpha_l$.

(c) The identities are Stewart's theorem for the cevians $q_2O$, $q_4O$, $q_1O$ and $q_3O$ of the triangles $q_1q_2q_3$, $q_1q_4q_3$, $q_2q_1q_4$ and $q_2q_3q_4$. Since $R_{12},R_{23}<R_{24}$, the first identity gives $r_{2O}^2+r_{1O}r_{3O}<R_{24}$, and the other inequalities follow in the same way. Since $R_{24}-r_{2O}^2=r_{4O}(r_{4O}+2r_{2O})$, the inequality $r_{1O}r_{3O}<R_{24}-r_{2O}^2$ gives $\psi<2+r_{4O}/r_{2O}$. In the same way $r_{4O}^2+r_{1O}r_{3O}<R_{24}$ gives $\psi<2+r_{2O}/r_{4O}$, hence $\psi<3$. The inequalities at $q_1$ and $q_3$ give $1/\psi<3$.

(d) Let $u_1$ and $u_2$ be the unit vectors with $q_1-O=r_{1O}u_1$ and $q_2-O=r_{2O}u_2$, so that $\ip{u_1}{u_2}=\cos\gamma$. In the triangle $q_2q_1q_4$, $\cot\alpha_1=(R_{12}+R_{14}-R_{24})/(4|A_3|)$. By the law of cosines, $R_{12}=r_{1O}^2+r_{2O}^2-2r_{1O}r_{2O}\cos\gamma$ and $R_{14}=r_{1O}^2+r_{4O}^2+2r_{1O}r_{4O}\cos\gamma$, so $R_{12}+R_{14}-R_{24}=2(r_{1O}^2-r_{2O}r_{4O})-2r_{1O}(r_{2O}-r_{4O})\cos\gamma$. With (b),
\begin{align*}
\cot\alpha_1&=\frac{r_{1O}-r_{2O}r_{4O}/r_{1O}-(r_{2O}-r_{4O})\cos\gamma}{r_{24}\sin\gamma},\\
\cot\alpha_3&=\frac{r_{3O}-r_{2O}r_{4O}/r_{3O}+(r_{2O}-r_{4O})\cos\gamma}{r_{24}\sin\gamma},
\intertext{and in the same way}
\cot\alpha_2&=\frac{r_{2O}-r_{1O}r_{3O}/r_{2O}-(r_{1O}-r_{3O})\cos\gamma}{r_{13}\sin\gamma},\\
\cot\alpha_4&=\frac{r_{4O}-r_{1O}r_{3O}/r_{4O}+(r_{1O}-r_{3O})\cos\gamma}{r_{13}\sin\gamma}.
\end{align*}
Since $r_{2O}r_{4O}/r_{1O}\pm r_{2O}r_{4O}/r_{3O}=(r_{3O}\pm r_{1O})/\psi$ and $r_{1O}r_{3O}/r_{2O}\pm r_{1O}r_{3O}/r_{4O}=\psi(r_{4O}\pm r_{2O})$, the sums and differences are as claimed.

(e) Take the vertex $q_1$; the others are similar. In the triangle $q_2q_1q_4$ the side $q_2q_4$ is the longest, so the angles at $q_2$ and $q_4$ are smaller than $\alpha_1$, and each of them exceeds $180^\circ-2\alpha_1$. If $\alpha_1\le90^\circ$, the angle $\phi$ at $q_4$ satisfies $180^\circ-2\alpha_1<\phi<\alpha_1\le90^\circ$, hence $\sin\phi>\sin2\alpha_1$, and the law of sines gives $r_{12}=r_{24}\sin\phi/\sin\alpha_1>2r_{24}\cos\alpha_1$. The same argument applies to $r_{14}$.

(f) Put $\phi_1=\angle q_3q_1q_4$, $\phi_3=\angle q_1q_3q_2$ and $\phi_4=\angle q_1q_4q_2$. The triangle $Oq_1q_4$ has the angles $\pi-\gamma$, $\phi_1$ and $\phi_4$, so $\gamma=\phi_1+\phi_4$. The triangle $q_1q_2q_3$ has the angles $\alpha_1-\phi_1$, $\alpha_2$ and $\phi_3$, so $\phi_1=\alpha_1+\alpha_2+\phi_3-\pi$. In the triangle $q_1q_2q_4$ the side $q_1q_2$, opposite $\phi_4$, is shorter than $q_2q_4$, which is opposite $\alpha_1$; so $\phi_4<\alpha_1$. In the same way $\phi_3<\alpha_2$. Hence $\gamma<2(\alpha_1+\alpha_2)-\pi$. If $90^\circ\le\alpha_1+\alpha_2\le180^\circ$, both angles lie in $[0,\pi]$, where the cosine is decreasing, and $\cos\gamma>\cos(2(\alpha_1+\alpha_2)-\pi)=2\sin^2(\alpha_1+\alpha_2)-1$.
\end{proof}

\begin{lemma}\label{lem:hN-dual}
Let $\{i,j\}$ and $\{o,l\}$ be complementary pairs of indices. Then $|M_2^{-1}(q_i-q_j)|=r_{ol}/(2|A_iA_j|)$.
\end{lemma}

\begin{proof}
Put $u=J(q_o-q_l)$, so that $\ip xu=\det(q_o-q_l,x)$ for $x\in\R^2$. By Lemma~\ref{lem:areas}(a),
\[
M_2u=\sum_pA_p\ip{q_p}{u}q_p=\sum_pA_p\ip{q_p-q_l}{u}q_p=\sum_pA_px_pq_p,\qquad x_p=\det(q_o-q_l,q_p-q_l).
\]
Here $x_o=x_l=0$, and $\sum_pA_px_p=\det\bigl(q_o-q_l,\sum_pA_p(q_p-q_l)\bigr)=0$, so $M_2u=A_ix_i(q_i-q_j)$. Hence $M_2^{-1}(q_i-q_j)=u/(A_ix_i)$. Finally, $|u|=r_{ol}$, and $|x_i|$ is twice the area of the triangle $q_oq_lq_i$, which is $2|A_j|$.
\end{proof}

\subsection{The bounds at the ends of the segment}\label{sec:hN-ends}

\begin{lemma}\label{lem:hN-vertex}
Let $\{i,o\}$ be one diagonal and $\{j,l\}$ the other. For every $t\in\R$,
\[
B(n_{io},t)=\frac13\sum_{v\in\{i,o\}}\ \sum_{u\in\{j,l\}}\frac{|A_v|}{|A_u|}\,\frac{1-t\,r_{vu}^3}{\sin^2\alpha_v}+d_{jl}(t)\,\bigl|M_2^{-1}(n_{jl}-n_{io})\bigr|^2 .
\]
\end{lemma}

\begin{proof}
The term of the edge $io$ vanishes. Let $vu$ be a side, with $v\in\{i,o\}$ and $u\in\{j,l\}$, and let $v'$ and $u'$ be the other ends of the diagonals through $q_v$ and $q_u$. Then $n_{vu}-n_{io}=q_u-q_{v'}$, and $\{u,v'\}$ is complementary to $\{v,u'\}$. Since $A_vA_u<0$, Lemma~\ref{lem:hN-dual} shows that the term of $vu$ is
\[
\tfrac13(1-t\,r_{vu}^3)\,|A_vA_u|\,R_{vu}\,\frac{R_{vu'}}{4A_u^2A_{v'}^2}=\frac{|A_v|}{3|A_u|}\,(1-t\,r_{vu}^3)\,\frac{R_{vu}R_{vu'}}{4A_{v'}^2}.
\]
The sides at $q_v$ are $vu$ and $vu'$, and $R_{vu}R_{vu'}\sin^2\alpha_v=4A_{v'}^2$ by Lemma~\ref{lem:hN-geom}(b).
\end{proof}

Let $h(x)=1-x^{3/2}$ for $x\ge0$; it is concave and decreasing. By Lemma~\ref{lem:hN-geom}(c) the numbers
\begin{align*}
V_{13}&=\frac{r_{24}}3\Bigl[\frac{h\bigl((r_{2O}^2+r_{1O}r_{3O})/R_{24}\bigr)}{r_{4O}}+\frac{h\bigl((r_{4O}^2+r_{1O}r_{3O})/R_{24}\bigr)}{r_{2O}}\Bigr],\\
V_{24}&=\frac{r_{13}}3\Bigl[\frac{h\bigl((r_{1O}^2+r_{2O}r_{4O})/R_{13}\bigr)}{r_{3O}}+\frac{h\bigl((r_{3O}^2+r_{2O}r_{4O})/R_{13}\bigr)}{r_{1O}}\Bigr]
\end{align*}
are positive. Put
\[
V_1=\frac{|A_1|}3\Bigl[\frac{h(R_{12}/R_{24})}{|A_2|}+\frac{h(R_{14}/R_{24})}{|A_4|}\Bigr].
\]
Let $g(x)=\sqrt x+1/(1+\sqrt x)$ for $x\ge0$, so that $h(x)=(1-x)\,g(x)$ and $g'(x)=(2+\sqrt x)/(2(1+\sqrt x)^2)$. For $0<r<3$ and $1\le\psi<3$ put
\[
V_0(r)=\tfrac43\,h\bigl(\tfrac{1+r}4\bigr),\qquad V^*(\psi)=V_0(1/\psi)+\tfrac1{12}\,g'\bigl(\tfrac{1+1/\psi}4\bigr)\bigl(1-\tfrac1\psi\bigr)\bigl(5-\tfrac1\psi\bigr).
\]
The function $V_0$ is decreasing, $V_0(1)=\frac{4-\sqrt2}3<1$, and $V_0(\psi_*)=\frac34$ for $\psi_*=4(7/16)^{2/3}-1$.

\begin{lemma}\label{lem:hN-jensen}
\begin{enumerate}[label=\textup{(\alph*)},leftmargin=2em]
\item $B(n_{13},r_{24}^{-3})=V_1\csc^2\alpha_1+V_3\csc^2\alpha_3$ and $B(n_{24},r_{13}^{-3})=V_2\csc^2\alpha_2+V_4\csc^2\alpha_4$, with $V_l\ge0$, $V_1+V_3\le V_{13}$ and $V_2+V_4\le V_{24}$.
\item If $\psi\ge1$, then $V_{13}\le V_0(\psi)$ and $V_{24}\le V^*(\psi)$.
\end{enumerate}
\end{lemma}

\begin{proof}
(a) Lemma~\ref{lem:hN-vertex} with $t=r_{24}^{-3}$, where $d_{24}(t)=0$ and $1-t\,r_{vu}^3=h(R_{vu}/R_{24})$, gives the first formula with $V_l=\frac13|A_l|\sum_{u\in\{2,4\}}h(R_{lu}/R_{24})/|A_u|$ for $l=1,3$. By Lemma~\ref{lem:hN-geom}(b),
\[
V_1+V_3=\frac{r_{24}}{3r_{13}}\Bigl[\frac{r_{3O}h(R_{12}/R_{24})+r_{1O}h(R_{23}/R_{24})}{r_{4O}}+\frac{r_{3O}h(R_{14}/R_{24})+r_{1O}h(R_{34}/R_{24})}{r_{2O}}\Bigr].
\]
The weights $r_{3O}/r_{13}$ and $r_{1O}/r_{13}$ have sum $1$, and $h$ is concave. By Jensen's inequality and the Stewart identities of Lemma~\ref{lem:hN-geom}(c), the first numerator is at most $r_{13}\,h((r_{2O}^2+r_{1O}r_{3O})/R_{24})$ and the second at most $r_{13}\,h((r_{4O}^2+r_{1O}r_{3O})/R_{24})$. So $V_1+V_3\le V_{13}$. The centre $n_{24}$ is treated in the same way, with the roles of the diagonals exchanged and $V_l=\frac13|A_l|\sum_{u\in\{1,3\}}h(R_{lu}/R_{13})/|A_u|$ for $l=2,4$.

(b) For $|\eta|<1$ and $0<r<3$ put $x_\pm=\frac14\bigl((1\pm\eta)^2+r(1-\eta^2)\bigr)$ and $\delta_\pm=(3-r)\pm(1-r)\eta$, so that $4(1-x_\pm)=(1\mp\eta)\delta_\pm$. We show that if $x_\pm\le1$, then
\begin{equation}\label{eq:hN-js}
\frac23\Bigl[\frac{h(x_+)}{1-\eta}+\frac{h(x_-)}{1+\eta}\Bigr]\le V_0(r)-\frac1{12}\,g'(x_0)\,(r-1)(5-r)\,\eta^2,\qquad x_0=\tfrac{1+r}4 .
\end{equation}
The function $g$ is concave, since $g''(x)=-(3+\sqrt x)/(4\sqrt x(1+\sqrt x)^3)$, so $g\le\ell$ with $\ell(x)=g(x_0)+g'(x_0)(x-x_0)$, and $h(x)\le(1-x)\ell(x)$ for $0\le x\le1$. Hence the left side of \eqref{eq:hN-js} is at most $\frac16[\delta_+\ell(x_+)+\delta_-\ell(x_-)]$. Here $\delta_++\delta_-=2(3-r)$, $x_\pm-x_0=\frac14(\pm2\eta+(1-r)\eta^2)$ and $\delta_+(x_+-x_0)+\delta_-(x_--x_0)=\frac12(1-r)(5-r)\eta^2$. Since $\frac13(3-r)\ell(x_0)=\frac43(1-x_0)g(x_0)=V_0(r)$, this gives \eqref{eq:hN-js}.

Put $\eta=(r_{2O}-r_{4O})/r_{24}$. Then $r_{2O}=\frac12r_{24}(1+\eta)$, $r_{4O}=\frac12r_{24}(1-\eta)$ and $r_{1O}r_{3O}=\psi\,r_{2O}r_{4O}=\frac14\psi R_{24}(1-\eta^2)$. Hence, with $r=\psi$, $x_+=(r_{2O}^2+r_{1O}r_{3O})/R_{24}$ and $x_-=(r_{4O}^2+r_{1O}r_{3O})/R_{24}$, and $V_{13}$ is the left side of \eqref{eq:hN-js}. As $1\le\psi<3$, \eqref{eq:hN-js} gives $V_{13}\le V_0(\psi)$. In the same way $V_{24}$ is the left side of \eqref{eq:hN-js} with $\eta=(r_{1O}-r_{3O})/r_{13}$ and $r=1/\psi\le1$. There $(1-r)(5-r)\ge0$ and $\eta^2<1$, so $V_{24}\le V^*(\psi)$.
\end{proof}

We may relabel the bodies and reflect the plane (Corollary~\ref{cor:midpoint-test}); the result is again a convex CC, and a reflection restores the counterclockwise order. A cyclic relabelling, in which the new $q_l$ is the old $q_{l+1}$, replaces $\psi$ by $1/\psi$. The relabelling $1\leftrightarrow3$ exchanges $\alpha_1$ and $\alpha_3$, so it keeps $\psi$ and $a_{24}$ and changes the sign of $a_{13}$. Likewise $2\leftrightarrow4$ keeps $\psi$ and $a_{13}$ and changes the sign of $a_{24}$. So we may assume
\begin{equation}\label{eq:hN-frame}
\psi\ge1,\qquad a_{13}\ge0,\qquad a_{24}\ge0 .
\end{equation}
Put $r_\circ=\min(r_{13},r_{24})$ and $t_\circ=r_\circ^{-3}$, and
\[
B_{13}^+=V_{13}\csc^2\alpha_3+V_1\,(\csc^2\alpha_1-\csc^2\alpha_3),\qquad B_{13}=V_{13}\csc^2\alpha_1,\qquad B_{24}=V_{24}\csc^2\alpha_4,
\]
\[
\hat B_{13}^+=B_{13}^++\Bigl(\frac{r_{24}^3}{r_\circ^3}-1\Bigr)(B_{13}^+-B_0),\qquad \hat B_{24}=B_{24}+\Bigl(\frac{r_{13}^3}{r_\circ^3}-1\Bigr)(B_{24}-B_0).
\]

\begin{lemma}\label{lem:hN-ends}
Assume \eqref{eq:hN-frame}. Then $60^\circ<\alpha_1\le90^\circ\le\alpha_4$, $\csc^2\alpha_1\ge\csc^2\alpha_3$ and $\csc^2\alpha_4\ge\csc^2\alpha_2$. Moreover $B(n_{13},t_\circ)\le\hat B_{13}^+$, $B(n_{24},t_\circ)\le\hat B_{24}$, and $\hat B_{13}^+\le B_{13}^+\le B_{13}<\frac43V_0(\psi)<\frac43$.
\end{lemma}

\begin{proof}
By Lemma~\ref{lem:hN-geom}(d), $p:=\frac12(\cot\alpha_1+\cot\alpha_3)\ge0$ and $p':=\frac12(\cot\alpha_2+\cot\alpha_4)\le0$. Hence $\cot\alpha_1=p+a_{13}\ge|p-a_{13}|=|\cot\alpha_3|$ and $-\cot\alpha_4=a_{24}-p'\ge|p'+a_{24}|=|\cot\alpha_2|$. With Lemma~\ref{lem:hN-geom}(a) this gives the statements on the angles.
By Lemma~\ref{lem:hN-jensen}(a), $B(n_{13},r_{24}^{-3})=(V_1+V_3)\csc^2\alpha_3+V_1(\csc^2\alpha_1-\csc^2\alpha_3)\le B_{13}^+$, and $V_1\le V_{13}$ gives $B_{13}^+\le B_{13}$. By Lemma~\ref{lem:hN-centre}(b), $B(n_{13},\cdot)$ is affine with value $B_0$ at $0$, and $t_\circ=(r_{24}/r_\circ)^3\,r_{24}^{-3}$ with $(r_{24}/r_\circ)^3\ge1$. So
\[
B(n_{13},t_\circ)=B(n_{13},r_{24}^{-3})+\Bigl(\frac{r_{24}^3}{r_\circ^3}-1\Bigr)\bigl(B(n_{13},r_{24}^{-3})-B_0\bigr)\le\hat B_{13}^+ .
\]
In the same way $B(n_{24},r_{13}^{-3})\le(V_2+V_4)\csc^2\alpha_4\le B_{24}$ and $B(n_{24},t_\circ)\le\hat B_{24}$. As $60^\circ<\alpha_1\le90^\circ$, we have $\csc^2\alpha_1<\frac43$, hence $B_{13}<\frac43V_{13}\le\frac43V_0(\psi)\le\frac43V_0(1)<\frac43$ by Lemma~\ref{lem:hN-jensen}(b). Finally $B_0\ge\frac43>B_{13}^+$ gives $\hat B_{13}^+\le B_{13}^+$.
\end{proof}

\subsection{The gain along the segment}\label{sec:hN-test}
Put
\[
E=W(t_\circ)\,\bigl|M_2^{-1}(n_{24}-n_{13})\bigr|^2 .
\]

\begin{lemma}\label{lem:hN-gain}
Assume \eqref{eq:hN-frame}. Then $E\ge\frac23(1+\cos\gamma)(a_{13}+a_{24})^2$. If moreover $\alpha_1+\alpha_2<180^\circ$, then
\[
\sqrt E>\frac1{\sqrt3}\,\sin^2(\alpha_1+\alpha_2)\Bigl[\frac1{\sin\alpha_1\sin\alpha_2}+\frac1{\sin\alpha_3\sin\alpha_4}\Bigr].
\]
\end{lemma}

\begin{proof}
As in the proof of Lemma~\ref{lem:hN-centre}(b), expanding $R_{ij}^2$ gives $\sum_eA_iA_jR_e^2=(\tr M_2)^2+2\tr(M_2^2)=3(\tr M_2)^2-4\det M_2$. Hence
\[
\sum_eA_iA_jr_e^5=r_\circ\bigl(3(\tr M_2)^2-4\det M_2\bigr)+\sum_eA_iA_jR_e^2\,(r_e-r_\circ),
\]
and every term of the last sum is nonnegative: $A_iA_j>0$ and $r_e\ge r_\circ$ on the diagonals, $A_iA_j<0$ and $r_e<r_\circ$ on the sides. With Lemma~\ref{lem:hN-centre}(c), $W(t_\circ)\ge(3(\tr M_2)^2-4\det M_2)/(3r_\circ^2)$.

Let $u_1,u_2$ be as in the proof of Lemma~\ref{lem:hN-geom}(d). By Lemma~\ref{lem:hN-geom}(b), the signs $A_1,A_3>0>A_2,A_4$ of Lemma~\ref{lem:areas}(d) and the translation invariance of $M_2$, $M_2=\sum_lA_l(q_l-O)(q_l-O)^{\mathsf T}=\zeta\,(r_{1O}r_{3O}\,u_1u_1^{\mathsf T}-r_{2O}r_{4O}\,u_2u_2^{\mathsf T})$ with $\zeta=\frac12r_{13}r_{24}\sin\gamma$. So $\tr M_2=\zeta(r_{1O}r_{3O}-r_{2O}r_{4O})$ and $\det M_2=-\zeta^2r_{1O}r_{2O}r_{3O}r_{4O}\sin^2\gamma$, and $3(\tr M_2)^2-4\det M_2=\zeta^2r_{1O}r_{2O}r_{3O}r_{4O}\bigl(3(\psi+1/\psi-2)+4\sin^2\gamma\bigr)$. Let $u_1^*,u_2^*$ be the dual basis, with $\ip{u_i^*}{u_j}=1$ for $i=j$ and $0$ otherwise. Then $M_2u_1^*=\zeta r_{1O}r_{3O}\,u_1$ and $M_2u_2^*=-\zeta r_{2O}r_{4O}\,u_2$. Since $n_{24}-n_{13}=(r_{2O}-r_{4O})u_2-(r_{1O}-r_{3O})u_1$, we get $M_2^{-1}(n_{24}-n_{13})=-(X_1u_1^*+X_2u_2^*)/\zeta$ with $X_1=(r_{1O}-r_{3O})/(r_{1O}r_{3O})$ and $X_2=(r_{2O}-r_{4O})/(r_{2O}r_{4O})$. As the Gram matrix of $u_1^*,u_2^*$ is the inverse of that of $u_1,u_2$,
\[
E\ge\frac{r_{1O}r_{2O}r_{3O}r_{4O}\bigl(3(\psi+1/\psi-2)+4\sin^2\gamma\bigr)}{3r_\circ^2\sin^2\gamma}\,\bigl(X_1^2-2X_1X_2\cos\gamma+X_2^2\bigr).
\]
By Lemma~\ref{lem:hN-geom}(d), $z_1:=2r_{24}\sin\gamma\,a_{13}=(r_{1O}r_{3O}+r_{2O}r_{4O})X_1-2r_{2O}r_{4O}X_2\cos\gamma$ and $z_2:=2r_{13}\sin\gamma\,a_{24}=(r_{1O}r_{3O}+r_{2O}r_{4O})X_2-2r_{1O}r_{3O}X_1\cos\gamma$. Expanding gives
\[
z_1^2+2z_1z_2\cos\gamma+z_2^2=\bigl((r_{1O}r_{3O}+r_{2O}r_{4O})^2-4r_{1O}r_{2O}r_{3O}r_{4O}\cos^2\gamma\bigr)\bigl(X_1^2-2X_1X_2\cos\gamma+X_2^2\bigr).
\]
Dividing the numerator and the denominator by $(r_{2O}r_{4O})^2$, we obtain
\[
E\ge\frac43\cdot\frac{3(\psi-1)^2+4\psi\sin^2\gamma}{(\psi-1)^2+4\psi\sin^2\gamma}\cdot\frac{R_{24}a_{13}^2+2r_{13}r_{24}a_{13}a_{24}\cos\gamma+R_{13}a_{24}^2}{r_\circ^2},
\]
where the middle factor is at least $1$. Finally $\xi_1^2+2\xi_1\xi_2\cos\gamma+\xi_2^2=\frac12(1+\cos\gamma)(\xi_1+\xi_2)^2+\frac12(1-\cos\gamma)(\xi_1-\xi_2)^2$ with $\xi_1=r_{24}a_{13}$ and $\xi_2=r_{13}a_{24}$, and $\xi_1+\xi_2\ge r_\circ(a_{13}+a_{24})$ because $a_{13},a_{24}\ge0$. This gives the first claim.

For the second, $\alpha_1+\alpha_2>120^\circ$ by Lemma~\ref{lem:hN-geom}(a), so Lemma~\ref{lem:hN-geom}(f) gives $1+\cos\gamma>2\sin^2(\alpha_1+\alpha_2)$. Since $\cot x+\cot y=\sin(x+y)/(\sin x\sin y)$ and $\alpha_3+\alpha_4=360^\circ-\alpha_1-\alpha_2$,
\[
a_{13}+a_{24}=\tfrac12\bigl[(\cot\alpha_1+\cot\alpha_2)-(\cot\alpha_3+\cot\alpha_4)\bigr]=\tfrac12\sin(\alpha_1+\alpha_2)\Bigl[\frac1{\sin\alpha_1\sin\alpha_2}+\frac1{\sin\alpha_3\sin\alpha_4}\Bigr],
\]
which is positive. Hence $E>\frac43\sin^2(\alpha_1+\alpha_2)(a_{13}+a_{24})^2$, which is the claim.
\end{proof}

\begin{proposition}\label{prop:hN-test}
Assume \eqref{eq:hN-frame}. Then $\mu_1<1$ if one of the following holds:
\begin{enumerate}[label=\textup{(\roman*)},leftmargin=2em]
\item $\hat B_{13}^+<1$;
\item $\hat B_{24}<1$;
\item $\hat B_{13}^+\ge1$, $\hat B_{24}\ge1$ and $\sqrt{\hat B_{13}^+-1}+\sqrt{\hat B_{24}-1}<\sqrt E$.
\end{enumerate}
\end{proposition}

\begin{proof}
For $\theta\in[0,1]$ put $y_\theta=(1-\theta)n_{13}+\theta n_{24}$. By Lemmas~\ref{lem:hN-centre}(c) and~\ref{lem:hN-ends},
\[
B(y_\theta,t_\circ)\le(1-\theta)\hat B_{13}^++\theta\hat B_{24}-\theta(1-\theta)E .
\]
By Lemma~\ref{lem:hN-centre}(d) with $t_0=t_\circ<\lambda'$, it suffices to find $\theta\in[0,1]$ for which the right side is less than $1$, since $1<\frac43$. In case (i) we take $\theta=0$, and in case (ii) $\theta=1$. In case (iii) put $p_1=\sqrt{\hat B_{13}^+-1}$, $p_2=\sqrt{\hat B_{24}-1}$ and $\theta=(p_1+\delta)/(p_1+p_2+2\delta)$ with $\delta>0$. Then the right side minus $1$ equals
\[
\theta(1-\theta)\Bigl[\frac{p_1^2}\theta+\frac{p_2^2}{1-\theta}-E\Bigr]\le\theta(1-\theta)\bigl[(p_1+p_2+2\delta)(p_1+p_2)-E\bigr],
\]
which is negative for small $\delta$.
\end{proof}

\subsection{Large \texorpdfstring{$\psi$}{psi}}\label{sec:hN-large}

\begin{proposition}\label{prop:hN-large}
Assume \eqref{eq:hN-frame} and $\psi\ge\psi_*$. Then $\hat B_{13}^+<1$.
\end{proposition}

\begin{proof}
By Lemma~\ref{lem:hN-ends}, $\hat B_{13}^+<\frac43V_0(\psi)\le\frac43V_0(\psi_*)=1$, because $V_0$ is decreasing.
\end{proof}

\subsection{The hard case}\label{sec:hN-hard}
In this subsection we assume
\begin{equation}\label{eq:hN-hard}
1\le\psi<\psi_*,\qquad a_{13}\ge0,\qquad a_{24}\ge0,\qquad \hat B_{13}^+\ge1,\qquad \hat B_{24}\ge1,
\end{equation}
and we prove that (iii) of Proposition~\ref{prop:hN-test} holds. Put $\alpha_\star=\arcsin\sqrt{V_0(1)}<90^\circ$,
\[
V_{13}'=V_0(\psi)+V_1\Bigl(\frac{\sin^2\alpha_3}{\sin^2\alpha_1}-1\Bigr),\qquad V_{24}'=V^*(\psi)+\Bigl(\frac{r_{13}^3}{r_\circ^3}-1\Bigr)\Bigl(V^*(\psi)-\tfrac43\sin^2\alpha_4\Bigr)_+,
\]
\[
\bar V=\tfrac12(V_{13}'+V_{24}'),\qquad k=-\cos(\alpha_1+\alpha_2),\qquad \chi=\cos(\bar\alpha_3-\bar\alpha_4).
\]

\begin{lemma}\label{lem:hN-eff}
Assume \eqref{eq:hN-hard}.
\begin{enumerate}[label=\textup{(\alph*)},leftmargin=2em]
\item $\hat B_{13}^+\le V_{13}'\csc^2\alpha_3$ and $\hat B_{24}\le V_{24}'\csc^2\alpha_4$.
\item $\sin^2\alpha_1\le V_0(\psi)$, $\sin^2\alpha_3\le V_{13}'$ and $\sin^2\alpha_4\le V_{24}'$. In particular $\alpha_1\le\alpha_\star$.
\end{enumerate}
\end{lemma}

\begin{proof}
(a) By Lemmas~\ref{lem:hN-ends} and~\ref{lem:hN-jensen}(b), $\hat B_{13}^+\le B_{13}^+\le V_0(\psi)\csc^2\alpha_3+V_1(\csc^2\alpha_1-\csc^2\alpha_3)=V_{13}'\csc^2\alpha_3$. Also $B_{24}\le V^*(\psi)\csc^2\alpha_4$, $r_{13}\ge r_\circ$ and $B_0\ge\frac43$, so
\[
\hat B_{24}\le V^*(\psi)\csc^2\alpha_4+\Bigl(\frac{r_{13}^3}{r_\circ^3}-1\Bigr)\Bigl(V^*(\psi)\csc^2\alpha_4-\frac43\Bigr)\le V_{24}'\csc^2\alpha_4 ,
\]
where the last step factors out $\csc^2\alpha_4$.
(b) The last two inequalities follow from $\hat B_{13}^+\ge1$ and $\hat B_{24}\ge1$ by (a). The first follows from $1\le\hat B_{13}^+\le B_{13}\le V_0(\psi)\csc^2\alpha_1$. Since $\alpha_1\le90^\circ$ and $V_0(\psi)\le V_0(1)$, it gives $\alpha_1\le\alpha_\star$.
\end{proof}

\begin{lemma}\label{lem:hN-two}
Assume \eqref{eq:hN-hard}, and put $z=\cos\alpha_1$ and $z_\star=\cos\alpha_\star=\sqrt{1-V_0(1)}$, so that $z_\star\le z<\frac12$.
\begin{enumerate}[label=\textup{(\alph*)},leftmargin=2em]
\item $R_{24}(1+8z^3)<2R_{13}$.
\item $V_0(\psi)+V^*(\psi)\le2V_0(1)+0.0166$.
\item $V_{13}'\le V_0(\psi)+0.035$.
\end{enumerate}
\end{lemma}

\begin{proof}
(a) By Lemma~\ref{lem:hN-geom}(e), $r_{12},r_{14}>2r_{24}z$. With $R_{12},R_{14}<R_{13}$ and $z\ge0$, the law of cosines in the triangle $q_2q_1q_4$ gives $R_{24}=R_{12}+R_{14}-2r_{12}r_{14}z<2R_{13}-8R_{24}z^3$.

(b) The function $x\mapsto x^{3/2}$ is convex, and its tangent at $\frac12$ gives $x_1^{3/2}+x_2^{3/2}\ge2^{-1/2}+\frac3{2\sqrt2}(x_1+x_2-1)$. With $x_1=\frac{1+\psi}4$ and $x_2=\frac{1+1/\psi}4$, so that $x_1+x_2-1=(\psi-1)^2/(4\psi)$, this gives $V_0(\psi)+V_0(1/\psi)\le2V_0(1)-(\psi-1)^2/(2\sqrt2\,\psi)$. The function $g'$ is decreasing, so $g'(\frac{1+1/\psi}4)\le g'(\frac{1+1/\psi_*}4)<0.4811$ by (N3). With $p=\psi-1\ge0$ we have $(1-\frac1\psi)(5-\frac1\psi)=p(4+5p)/(1+p)^2\le p(4+p)/(1+p)$. Hence
\[
V_0(\psi)+V^*(\psi)-2V_0(1)\le\frac p{1+p}\Bigl(\frac{0.4811}3-\Bigl(\frac1{2\sqrt2}-\frac{0.4811}{12}\Bigr)p\Bigr).
\]
By (N4), $\frac1{2\sqrt2}-\frac{0.4811}{12}\ge0.31346$, and the quadratic $0.31346\,p^2+(0.0166-\frac{0.4811}3)\,p+0.0166$ has a negative discriminant. So the right side is at most $0.0166$ for all $p\ge0$.

(c) Put $\varepsilon_1=V_{13}'-V_0(\psi)=V_1(\sin^2\alpha_3/\sin^2\alpha_1-1)$, which is nonnegative by Lemma~\ref{lem:hN-ends}. By Lemma~\ref{lem:hN-geom}(b),(e) and $r_{3O}/(r_{2O}r_{4O})=\psi/r_{1O}$,
\[
V_1=\frac{r_{24}r_{3O}}{3r_{13}r_{2O}r_{4O}}\Bigl(r_{2O}\,h(R_{12}/R_{24})+r_{4O}\,h(R_{14}/R_{24})\Bigr)<\frac{\psi\,R_{24}}{3r_{13}r_{1O}}\,(1-8z^3).
\]
The length $r_{1O}$ is at least the distance from $q_1$ to the line $q_2q_4$, which is $r_{12}r_{14}\sin\alpha_1/r_{24}>4r_{24}z^2\sin\alpha_1$. With (a), $\psi<\psi_*$ and $1-8z^3=2(\frac12-z)(1+2z+4z^2)$,
\[
V_1<\frac{\psi\,r_{24}\,(1-8z^3)}{12\,r_{13}\,z^2\sqrt{1-z^2}}<2\psi_*\sqrt{\frac2{1+8z^3}}\;H(z)\Bigl(\frac12-z\Bigr),\qquad H(z)=\frac{1+2z+4z^2}{12z^2\sqrt{1-z^2}}.
\]
The function $\sqrt{2/(1+8z^3)}$ is decreasing, and so is $H$ on $(0,\frac12]$ by (N6). Hence $V_1<4.64\,(\frac12-z)$ by (N5). As $\sin^2\alpha_3\le V_{13}'=V_0(\psi)+\varepsilon_1$ by Lemma~\ref{lem:hN-eff}(b),
\[
\varepsilon_1\le V_1\,\frac{V_0(\psi)+\varepsilon_1-\sin^2\alpha_1}{\sin^2\alpha_1},\qquad\text{so}\qquad \varepsilon_1\bigl(\sin^2\alpha_1-V_1\bigr)\le V_1\bigl(V_0(\psi)-\sin^2\alpha_1\bigr).
\]
The function $1-z^2-4.64(\frac12-z)$ is increasing in $z$ and is at least $0.266$ at $z=z_\star$ by (N8), so $\sin^2\alpha_1-V_1>0.266$. Also $0\le V_0(\psi)-\sin^2\alpha_1\le V_0(1)-1+z^2=z^2-z_\star^2\le(z-z_\star)(\frac12+z_\star)$, by Lemma~\ref{lem:hN-eff}(b). Since $x\mapsto x/(\sin^2\alpha_1-x)$ is increasing on $[0,\sin^2\alpha_1)$,
\[
\varepsilon_1\le\frac{4.64\,(\frac12-z)(z-z_\star)(\frac12+z_\star)}{1-z^2-4.64\,(\frac12-z)}\le0.035,
\]
where the last inequality holds for every $z\in[z_\star,\frac12]$ by (N5) and (N7).
\end{proof}

\begin{lemma}\label{lem:hN-angles}
Assume \eqref{eq:hN-hard}. Then $\alpha_3>90^\circ$, $\alpha_4>90^\circ$, $120^\circ<\alpha_1+\alpha_2<180^\circ$ and $\bar\alpha_3>60^\circ$. If $V^*(\psi)>\frac43\sin^2\alpha_4$, then $\bar\alpha_4<60^\circ<\alpha_1$ and $(r_{13}/r_\circ)^3\le\psi^{3/2}$.
\end{lemma}

\begin{proof}
By Lemma~\ref{lem:hN-ends}, $\alpha_4\ge90^\circ$. If $\alpha_4=90^\circ$, then $\cot\alpha_4=p'-a_{24}$ vanishes with $p'\le0\le a_{24}$ (notation of that proof), so $p'=0$ and $\psi=1$ by Lemma~\ref{lem:hN-geom}(d). Then $V_{24}'=V^*(1)=V_0(1)<1=\sin^2\alpha_4$, contradicting Lemma~\ref{lem:hN-eff}(b). Therefore $\alpha_4>90^\circ$.

Put $\delta=180^\circ-\alpha_1-\alpha_3$. By Lemma~\ref{lem:hN-geom}(d), $\sin(\alpha_1+\alpha_3)=\sin\alpha_1\sin\alpha_3\,(\cot\alpha_1+\cot\alpha_3)\ge0$; as $120^\circ<\alpha_1+\alpha_3<270^\circ$, we get $\delta\ge0$, and $\delta<60^\circ$ by Lemma~\ref{lem:hN-geom}(a). Since $\alpha_2+\alpha_4=180^\circ+\delta$, Lemma~\ref{lem:hN-geom}(d) gives
\[
\frac{\sin\delta}{2\sin\alpha_1\sin\alpha_3}=\frac{r_{13}(\psi-1)}{2\psi\,r_{24}\sin\gamma},\qquad\frac{\sin\delta}{2\sin\alpha_2\sin\alpha_4}=\frac{r_{24}(\psi-1)}{2r_{13}\sin\gamma},
\]
and multiplying, $\sin^2\delta=(\psi-1)^2\sin\alpha_1\sin\alpha_2\sin\alpha_3\sin\alpha_4/(\psi\sin^2\gamma)\le(\psi-1)^2/(\psi\sin^2\gamma)$. Next, $r_{24}\sin\gamma$ is the sum of the distances from $q_2$ and $q_4$ to the line $q_1q_3$, that is, $r_{12}\sin\angle q_2q_1q_3+r_{14}\sin\angle q_3q_1q_4$. The two angles are acute with sum $\alpha_1$, and $\sin x+\sin y\ge\sin(x+y)$ for $x,y\in[0^\circ,90^\circ]$, so Lemma~\ref{lem:hN-geom}(e) gives $r_{24}\sin\gamma\ge\min(r_{12},r_{14})\sin\alpha_1>r_{24}\sin2\alpha_1$. Since $120^\circ<2\alpha_1\le2\alpha_\star<180^\circ$, we have $\sin2\alpha_1\ge\sin2\alpha_\star=2\sqrt{V_0(1)(1-V_0(1))}$. As $(\psi-1)^2/\psi$ is increasing for $\psi\ge1$,
\[
\sin^2\delta<\frac{(\psi_*-1)^2}{4\psi_*V_0(1)(1-V_0(1))}\le0.14996
\]
by (N15). Let $\hat\delta\in(0^\circ,90^\circ)$ with $\sin^2\hat\delta=0.14996$; then $\delta<\hat\delta$. Suppose that $\alpha_3\le90^\circ$. Then $90^\circ\le180^\circ-\alpha_3=\alpha_1+\delta<\alpha_\star+\hat\delta<180^\circ$, and $\sin\alpha_3=\sin(\alpha_1+\delta)$. As the sine is decreasing on $[90^\circ,180^\circ]$,
\[
\sin\alpha_3\ge\sin(\alpha_\star+\hat\delta)=\sqrt{V_0(1)(1-0.14996)}+\sqrt{(1-V_0(1))\,0.14996}>\sqrt{V_0(1)+0.035}
\]
by (N15). But $\sin^2\alpha_3\le V_{13}'\le V_0(\psi)+0.035\le V_0(1)+0.035$ by Lemmas~\ref{lem:hN-eff}(b) and~\ref{lem:hN-two}(c). Hence $\alpha_3>90^\circ$.

Now $\alpha_1+\alpha_2>120^\circ$ by Lemma~\ref{lem:hN-geom}(a), and $\alpha_1+\alpha_2=\bar\alpha_3+\bar\alpha_4<180^\circ$. Also $\bar\alpha_3=\alpha_1+\delta>60^\circ$.

Suppose that $V^*(\psi)>\frac43\sin^2\alpha_4$. The function $V^*$ is increasing, because $V_0$ and $g'$ are decreasing and $(1-r)(5-r)$ is positive and decreasing for $r<1$. So $\sin^2\bar\alpha_4<\frac34V^*(\psi)\le\frac34V^*(\psi_*)<\frac34$, because $V^*(\psi_*)\le0.98201<1$ by (N12), and $\bar\alpha_4<60^\circ<\alpha_1$. By Lemma~\ref{lem:hN-geom}(b), $\sin\alpha_1\sin\alpha_3/(\sin\alpha_2\sin\alpha_4)=|A_1A_3|/|A_2A_4|=\psi R_{24}/R_{13}$. Moreover $2\sin\alpha_1\sin\alpha_3=\cos\delta-\cos(\alpha_1+\bar\alpha_3)$ and $2\sin\alpha_2\sin\alpha_4=\cos\delta-\cos(\alpha_2+\bar\alpha_4)$, because $\bar\alpha_3-\alpha_1=\alpha_2-\bar\alpha_4=\delta$. Here $(\alpha_1+\bar\alpha_3)-(\alpha_2+\bar\alpha_4)=2(\alpha_1-\bar\alpha_4)>0$, and both sums lie in $(0^\circ,180^\circ)$, since $\alpha_3>90^\circ\ge\alpha_1$. Hence $\sin\alpha_2\sin\alpha_4\le\sin\alpha_1\sin\alpha_3$, so $R_{13}\le\psi R_{24}$, and $(r_{13}/r_\circ)^3=\max(1,(r_{13}/r_{24})^3)\le\psi^{3/2}$.
\end{proof}

\begin{lemma}\label{lem:hN-const}
Assume \eqref{eq:hN-hard}. Then $V_{24}'\le V^*(\psi)+0.034$, $\bar V\le0.905$ and $\chi\ge0.92$.
\end{lemma}

\begin{proof}
If $V^*(\psi)\le\frac43\sin^2\alpha_4$, then $V_{24}'=V^*(\psi)$. Otherwise $(r_{13}/r_\circ)^3-1\le\psi^{3/2}-1$ and $\bar\alpha_4<60^\circ$ by Lemma~\ref{lem:hN-angles}. Put $x=V_0(\psi)\in(\frac34,V_0(1)]$ and $X=x+0.035$, and let $\vartheta_X\in(0^\circ,90^\circ)$ with $\sin^2\vartheta_X=X$. The angle $\bar\alpha_3$ is acute, with $\sin^2\bar\alpha_3=\sin^2\alpha_3\le V_{13}'\le X$ by Lemmas~\ref{lem:hN-eff}(b) and~\ref{lem:hN-two}(c), so $\bar\alpha_3\le\vartheta_X$, and $\bar\alpha_4=\alpha_1+\alpha_2-\bar\alpha_3>120^\circ-\vartheta_X>30^\circ$. Hence
\[
\sin^2\alpha_4=\sin^2\bar\alpha_4>\sin^2(120^\circ-\vartheta_X)=\tfrac34-\tfrac X2+\tfrac{\sqrt3}2\sqrt{X(1-X)} .
\]
By Lemma~\ref{lem:hN-two}(b), $V^*(\psi)\le2V_0(1)+0.0166-x$. Therefore
\begin{align*}
V_{24}'-V^*(\psi)&=\Bigl(\frac{r_{13}^3}{r_\circ^3}-1\Bigr)\bigl(V^*(\psi)-\tfrac43\sin^2\alpha_4\bigr)\le(\psi^{3/2}-1)\,b(x),\\
b(x)&=2V_0(1)+0.0166-x-\tfrac43\Bigl(\tfrac34-\tfrac X2+\tfrac{\sqrt3}2\sqrt{X(1-X)}\Bigr).
\end{align*}
Here $\psi=4(1-\frac34x)^{2/3}-1$ is decreasing in $x$, and $b$ is increasing on $[\frac34,V_0(1)]$ by (N10). For $x\le\frac45$ the product is at most $(\psi_*^{3/2}-1)\,b(\frac45)$, and for $x\ge\frac45$ it is at most $((4(\frac25)^{2/3}-1)^{3/2}-1)\,b(V_0(1))$. Both are at most $0.034$ by (N9).

By Lemma~\ref{lem:hN-two}(b),(c), $\bar V\le\frac12\bigl(V_0(\psi)+V^*(\psi)+0.069\bigr)\le V_0(1)+0.0083+0.0345\le0.905$ by (N11).

Let $\vartheta'\in(60^\circ,90^\circ)$ with $\sin^2\vartheta'=V_{13}'$; it exists because $\frac34<V_0(\psi)\le V_{13}'\le V_0(1)+0.035<1$. If $\bar\alpha_3\ge\bar\alpha_4$, then $0\le\bar\alpha_3-\bar\alpha_4=2\bar\alpha_3-(\alpha_1+\alpha_2)<2\vartheta'-120^\circ<60^\circ$, so
\[
\chi>\cos(2\vartheta'-120^\circ)=V_{13}'-\tfrac12+\sqrt{3V_{13}'(1-V_{13}')} .
\]
The function $u\mapsto u-\frac12+\sqrt{3u(1-u)}$ is decreasing on $(\frac34,1)$, because its derivative has the sign of $-(4u-1)(4u-3)$. So $\chi$ exceeds its value at $u=V_0(1)+0.035$, a point of $(\frac34,1)$ by (N1), and this value is at least $0.92$ by (N12). If $\bar\alpha_4>\bar\alpha_3$, let $\vartheta^*\in(60^\circ,90^\circ]$ with $\sin^2\vartheta^*=V^*(\psi)$; it exists because $V^*(\psi)\ge V_0(1/\psi)\ge V_0(1)>\frac34$ and $V^*(\psi)\le V^*(\psi_*)\le0.98201<1$ by (N12). Then $\bar\alpha_4\le\vartheta^*$: if $V^*(\psi)\le\frac43\sin^2\alpha_4$, then $V_{24}'=V^*(\psi)$ and $\sin^2\bar\alpha_4\le V_{24}'$ by Lemma~\ref{lem:hN-eff}(b); otherwise $\sin^2\bar\alpha_4<\frac34V^*(\psi)$. As $\bar\alpha_3>60^\circ$, we get $0<\bar\alpha_4-\bar\alpha_3<\vartheta^*-60^\circ\le30^\circ$, and
\[
\chi>\cos(\vartheta^*-60^\circ)=\tfrac12\sqrt{1-V^*(\psi)}+\tfrac{\sqrt3}2\sqrt{V^*(\psi)} .
\]
The right side is decreasing in $V^*(\psi)$ on $(\frac34,1)$, and $V^*(\psi)\le V^*(\psi_*)$, so $\chi$ exceeds its value at $V^*(\psi_*)$, which is at least $0.92$ by (N12).
\end{proof}

\begin{lemma}\label{lem:hN-form}
Assume \eqref{eq:hN-hard}. If
\begin{equation}\label{eq:hN-AI}
(1+k\chi)(2\bar V-1-k\chi)<\tfrac13(1-k)^2(1+2k+\chi)^2,
\end{equation}
then $\sqrt{\hat B_{13}^+-1}+\sqrt{\hat B_{24}-1}<\sqrt E$.
\end{lemma}

\begin{proof}
By Lemma~\ref{lem:hN-eff},
\[
\sqrt{\hat B_{13}^+-1}\le\frac{\sqrt{V_{13}'-\sin^2\alpha_3}}{\sin\alpha_3},\qquad\sqrt{\hat B_{24}-1}\le\frac{\sqrt{V_{24}'-\sin^2\alpha_4}}{\sin\alpha_4}.
\]
So the left side is at most $Z/(\sin\alpha_3\sin\alpha_4)$, where, by the Cauchy--Schwarz inequality,
\[
Z=\sin\alpha_3\sqrt{V_{24}'-\sin^2\alpha_4}+\sin\alpha_4\sqrt{V_{13}'-\sin^2\alpha_3}\le\sqrt{Y(2\bar V-Y)},\qquad Y=\sin^2\alpha_3+\sin^2\alpha_4 .
\]
Since $\bar\alpha_3+\bar\alpha_4=\alpha_1+\alpha_2$, the identity $\sin^2x+\sin^2y=1-\cos(x+y)\cos(x-y)$ gives $Y=1+k\chi$. Also $2\sin\alpha_3\sin\alpha_4=\cos(\bar\alpha_3-\bar\alpha_4)-\cos(\bar\alpha_3+\bar\alpha_4)=\chi+k$ and $2\sin\alpha_1\sin\alpha_2=\cos(\alpha_1-\alpha_2)+k\le1+k$, so $N:=\sin\alpha_3\sin\alpha_4/(\sin\alpha_1\sin\alpha_2)\ge(\chi+k)/(1+k)$. By Lemmas~\ref{lem:hN-gain} and~\ref{lem:hN-angles}, and $\sin^2(\alpha_1+\alpha_2)=1-k^2$,
\[
\sqrt E>\frac{(1-k^2)(1+N)}{\sqrt3\,\sin\alpha_3\sin\alpha_4}\ge\frac{(1-k)(1+2k+\chi)}{\sqrt3\,\sin\alpha_3\sin\alpha_4}.
\]
Hence the claim holds if $Y(2\bar V-Y)\le\frac13(1-k)^2(1+2k+\chi)^2$, which is \eqref{eq:hN-AI}.
\end{proof}

\begin{lemma}\label{lem:hN-quartic}
Under \eqref{eq:hN-hard}, inequality \eqref{eq:hN-AI} holds.
\end{lemma}

\begin{proof}
By Lemmas~\ref{lem:hN-angles} and~\ref{lem:hN-const}, $\frac12<k<1$, $0.92\le\chi\le1$ and $\bar V\le0.905$. The left side of \eqref{eq:hN-AI} is $(1+X)(2\bar V-1-X)$ with $X=k\chi\in(0,1)$. It is increasing in $\bar V$, and decreasing in $X$ because $\bar V<1$; the right side is increasing in $\chi$. So it suffices to take $\bar V=0.905$ and $\chi=0.92$, that is, to show
\[
F(k)=\bigl[(1-k)(1.92+2k)\bigr]^2-3(1+0.92k)(0.81-0.92k)>0\qquad\text{for }\tfrac12<k<1 .
\]
For $k_*=0.81/0.92\le k<1$ the first term is positive and the subtracted term is not. On $[\frac12,k_*]$ the function $F''(k)=48k^2-1.92k-10.2688$ is increasing, and $F''(\frac12)=0.7712>0$ by (N13), so $F$ is convex there. Its tangent at $0.79$ is $T(k)=F(0.79)+F'(0.79)(k-0.79)$, with $F(0.79)=\frac{2730393}{25000000}$ and $F'(0.79)=\frac{273}{31250}$, and $T$ is positive at $\frac12$ and at $k_*$ by (N14). Hence $F\ge T>0$ on $[\frac12,k_*]$.
\end{proof}

\subsection{Proof of Theorem~\ref{thm:handN}}
By Corollary~\ref{cor:midpoint-test} we may relabel the bodies so that \eqref{eq:hN-frame} holds (Section~\ref{sec:hN-ends}). If $\psi\ge\psi_*$, Propositions~\ref{prop:hN-large} and~\ref{prop:hN-test}(i) give $\mu_1<1$. Let $1\le\psi<\psi_*$. If $\hat B_{13}^+<1$ or $\hat B_{24}<1$, Proposition~\ref{prop:hN-test} applies. Otherwise \eqref{eq:hN-hard} holds, Lemmas~\ref{lem:hN-quartic} and~\ref{lem:hN-form} give (iii) of Proposition~\ref{prop:hN-test}, and again $\mu_1<1$. \qed

\section{Proof of Theorem~A}\label{sec:proofs}

\begin{corollary}\label{cor:nondeg}
At every convex CC, $Q(v)\ge(1-\lmax(G))\,K(v)\ge0$ for all $v$, and $Q(v)>0$ for $v\notin\cT$. In particular every convex CC is a nondegenerate local minimum of $f_m$, and its Morse index is $0$.
\end{corollary}

\begin{proof}
By Theorem~\ref{thm:handN}, $\mu_1=\lmax(G)<1$. By Proposition~\ref{prop:majorant}(c), $Q(v)\ge(1-\mu_1)K(v)$ for all $v$. By Lemma~\ref{lem:rigid}, $K(v)>0$ for $v\notin\cT$, so $Q(v)>0$ for $v\notin\cT$, and Lemma~\ref{lem:shape} gives the last statement.
\end{proof}

Let $\mathcal M=\{m\in(0,\infty)^4:\ m_1+m_2+m_3+m_4=1\}$, an open convex subset of an affine hyperplane, and let
\[
\cS^+=\{[q]\in\cS:\ A_1(q)>0,\ A_2(q)<0,\ A_3(q)>0,\ A_4(q)<0\}.
\]
By Lemma~\ref{lem:areas}(b) the signs of the $A_l$ depend only on the class $[q]$, so $\cS^+$ is a well-defined open subset of $\cS$. By Lemma~\ref{lem:areas}(c) it is disjoint from $\Delta$. If all $A_l$ are nonzero, then no three of the $q_l$ are collinear. Hence, by Lemma~\ref{lem:areas}(d), $\cS^+$ is the set of classes of counterclockwise strictly convex quadrilaterals with vertices $q_1,q_2,q_3,q_4$ in this order. Complex conjugation induces a real-analytic involution $\iota[q]=[\bar q]$ of $\cS$, which is well defined because $\overline{\mu q+t\mathbf 1}=\bar\mu\bar q+\bar t\mathbf 1$. It preserves $\Delta$ and satisfies $f_m\circ\iota=f_m$. By Lemma~\ref{lem:areas}(b) it maps $\cS^+$ onto the set where all four signs are reversed, which is disjoint from $\cS^+$.

The function $f(s,m)=f_m(s)$ is real-analytic on $(\cS\setminus\Delta)\times(0,\infty)^4$. Let
\[
\cZ=\{(s,m)\in\cS^+\times\mathcal M:\ s\text{ is a critical point of }f_m\},
\]
and let $\mathrm{pr}:\cZ\to\mathcal M$ be the projection. By Section~\ref{sec:shape}, $\mathrm{pr}^{-1}(m)$ is in bijection with the set of classes of counterclockwise convex CCs with cyclic order $(1234)$ for the masses $m$.

\begin{lemma}\label{lem:local}
$\cZ$ is a real-analytic submanifold of $\cS^+\times\mathcal M$ of dimension $3$, and $\mathrm{pr}$ is a local real-analytic diffeomorphism.
\end{lemma}

\begin{proof}
Let $(s_0,m_0)\in\cZ$, and choose a real-analytic chart $\chi$ of $\cS^+$ near $s_0$. Near $(s_0,m_0)$, the set $\cZ$ is the zero set of the map $(u,m)\mapsto\nabla_uf(\chi^{-1}(u),m)\in\R^4$. Its derivative with respect to $u$ at $(\chi(s_0),m_0)$ is the matrix of $D^2f_{m_0}(s_0)$ in the chart. This matrix is invertible by Corollary~\ref{cor:nondeg}, since $s_0$ is the class of a convex CC. By the real-analytic implicit function theorem there are neighbourhoods $\mathcal U$ of $s_0$ and $\mathcal N$ of $m_0$ and a real-analytic map $g:\mathcal N\to\mathcal U$ such that $\cZ\cap(\mathcal U\times\mathcal N)$ is the graph of $g$. The claims follow.
\end{proof}

\begin{lemma}\label{lem:shub}
Let $m^n\in(0,\infty)^4$ converge to $m^\infty\in(0,\infty)^4$. For each $n$ let $q^n$ be a CC for the masses $m^n$ with centre of mass $0$ and $I(q^n)=1$, both computed with the masses $m^n$. Then a subsequence of $(q^n)$ converges to a CC $q^\infty$ for the masses $m^\infty$. In particular $q^\infty$ is collision-free.
\end{lemma}

For fixed masses this is Shub's lemma \cite{Shu71}, which says that the CCs stay away from the collisions. We follow the proof in \cite[Proposition~15]{Moe15}, and include it because here the masses vary.

\begin{proof}
Write $U_n$ for the potential with the masses $m^n$, and $\lambda_n=U_n(q^n)$ for the multiplier of $q^n$. Since $m^n_i|q^n_i|^2\le I(q^n)=1$ and $m^n_i\to m^\infty_i>0$, the sequence is bounded, and after passing to a subsequence $q^n\to q^\infty$. By continuity, $q^\infty$ has centre of mass $0$ and $I(q^\infty)=1$ for the masses $m^\infty$. In particular the $q^\infty_i$ are not all equal. Suppose that $q^\infty$ has a collision. Group the indices into clusters $\Gamma$, consisting of indices with a common limit position $p_\Gamma$. Some $r^n_{ij}$ tends to $0$, so $\lambda_n\to\infty$. Sum \eqref{eq:cc} over $i\in\Gamma$. The mutual forces between the bodies of $\Gamma$ cancel in pairs, and the forces exerted by the other bodies stay bounded, because their distances to the bodies of $\Gamma$ have positive limits. Hence $\lambda_n\sum_{i\in\Gamma}m^n_iq^n_i$ is bounded, and $\sum_{i\in\Gamma}m^n_iq^n_i\to0$. The limit of this sum is $(\sum_{i\in\Gamma}m^\infty_i)\,p_\Gamma$, so $p_\Gamma=0$ for every cluster $\Gamma$. Different clusters have different limit positions, so there is only one cluster, and $q^\infty_i=0$ for all $i$. This is a contradiction, so $q^\infty$ is collision-free. Then $\lambda_n\to U_\infty(q^\infty)$, where $U_\infty$ is the potential with the masses $m^\infty$, and letting $n\to\infty$ in \eqref{eq:cc} shows that $q^\infty$ is a CC for $m^\infty$.
\end{proof}

The next lemma is a second application of Corollary~\ref{cor:nondeg}.

\begin{lemma}\label{lem:collinear}
Let $q^n$ be convex CCs for masses $m^n\in(0,\infty)^4$. Suppose that $m^n\to m\in(0,\infty)^4$ and that $q^n$ converges to a CC $q$ for the masses $m$. Then $q$ is not collinear.
\end{lemma}

\begin{proof}
Suppose that $q$ lies on a line with unit direction $u$. Let $k$ be the body with the largest coordinate $\ip{q_j}{u}$, write $r_j=r_{kj}$, and let all sums run over $j\ne k$. Then $q_k-q_j=r_ju$ for $j\ne k$, and the three distances $r_j$ are distinct, because $q$ is collision-free. Let $v\in(\R^2)^4$ move $q_k$ by $Ju$ and leave the other bodies fixed. Then $\dr_e(v)=0$ for every edge $e$, so $K(v)=0$, and Lemma~\ref{lem:hessian}(a) gives
\[
Q(v)=-\sum\omega_{kj}=-m_k\sum m_j\,(s_{kj}-\lambda').
\]
The component of \eqref{eq:equil} for $i=k$ along $u$ is $\sum m_j(s_{kj}-\lambda')r_j=0$. Hence $\lambda'=\sum m_jr_j^{-2}/\sum m_jr_j$, and
\begin{align*}
\Bigl(\sum m_jr_j\Bigr)\sum m_j(s_{kj}-\lambda')&=\Bigl(\sum m_jr_j^{-3}\Bigr)\Bigl(\sum m_jr_j\Bigr)-\Bigl(\sum m_jr_j^{-2}\Bigr)\Bigl(\sum m_j\Bigr)\\
&=\frac12\sum_{i,j\ne k}m_im_j\,(r_j-r_i)(r_i^{-3}-r_j^{-3}).
\end{align*}
The terms of the last sum are nonnegative, and they are positive when $r_i\ne r_j$, so $Q(v)<0$. On the other hand, let $Q_n$ and $K_n$ be the forms $Q$ and $K$ of the CC $q^n$ for the masses $m^n$. By Corollary~\ref{cor:nondeg}, $Q_n(v)\ge0$. At collision-free configurations the multiplier $\lambda=U/I$ depends continuously on the configuration and the masses, and hence so do $\kappa_e$, $\omega_e$ and $\dr_e(v)$ in Lemma~\ref{lem:hessian}(a). Therefore $Q_n(v)\to Q(v)$, and $Q(v)\ge0$, a contradiction.
\end{proof}

\begin{lemma}\label{lem:proper}
The projection $\mathrm{pr}:\cZ\to\mathcal M$ is proper.
\end{lemma}

\begin{proof}
Let $C\subset\mathcal M$ be compact, and let $(s_n,m_n)$ be a sequence in $\mathrm{pr}^{-1}(C)$. After passing to a subsequence, $m_n\to m^\infty\in C$. Represent $s_n$ by a CC $q^n$ for $m_n$ with centre of mass $0$ and $I(q^n)=1$. By Lemma~\ref{lem:shub}, after passing to a further subsequence, $q^n$ converges to a CC $q^\infty$ for $m^\infty$, and $s_n\to s_\infty=[q^\infty]\in\cS\setminus\Delta$. Since $s_n\in\cS^+$, the $q^n$ are convex CCs, so $q^\infty$ is not collinear by Lemma~\ref{lem:collinear}. By Lemma~\ref{lem:nocollinear} no three of the $q^\infty_l$ are collinear, and by Lemma~\ref{lem:areas}(d) all $A_l(q^\infty)$ are nonzero. As limits of the $A_l(q^n)$ they have the signs $(+,-,+,-)$. Hence $s_\infty\in\cS^+$, and $(s_\infty,m^\infty)\in\mathrm{pr}^{-1}(C)$. Every sequence in $\mathrm{pr}^{-1}(C)$ therefore has a subsequence that converges in $\mathrm{pr}^{-1}(C)$. Since $\cZ$ is metrizable, $\mathrm{pr}^{-1}(C)$ is compact.
\end{proof}

\begin{lemma}\label{lem:count}
Let $p:X\to Y$ be a proper local homeomorphism between metrizable spaces, with $Y$ connected. Then all fibres of $p$ are finite and have the same number of points.
\end{lemma}

\begin{proof}
A fibre $p^{-1}(y)$ is compact because $p$ is proper, and discrete because $p$ is a local homeomorphism. It is therefore finite, say $p^{-1}(y)=\{x_1,\dots,x_k\}$. Choose pairwise disjoint open sets $U_i\ni x_i$ such that $p$ maps $U_i$ homeomorphically onto an open set $V_i\ni y$, and put $F=X\setminus\bigcup_iU_i$. The set $p(F)$ is closed. Indeed, if $p(x_n)\to z$ with $x_n\in F$, then the $x_n$ lie in the compact set $p^{-1}(\{z\}\cup\{p(x_n):n\ge1\})$, so a subsequence converges to a point $x\in F$ with $p(x)=z$. Since also $y\notin p(F)$, the set $W=\bigcap_iV_i\setminus p(F)$ is an open neighbourhood of $y$ (if $k=0$, then $W=Y\setminus p(F)$). For $z\in W$ the fibre $p^{-1}(z)$ lies in $\bigcup_iU_i$ and meets each $U_i$ in exactly one point, so it has $k$ points. Hence the number of points in the fibres is locally constant, and therefore constant.
\end{proof}

\begin{proof}[Proof of Theorem~A]
By Lemmas~\ref{lem:local}, \ref{lem:proper} and~\ref{lem:count}, the number $N(m)$ of points of $\mathrm{pr}^{-1}(m)$ is finite and independent of $m\in\mathcal M$. For four equal masses, the only convex CC up to similarity is the square. Indeed, let $q$ be such a CC with cyclic order $(1234)$. By \cite[Theorem~1]{AFS08}, applied to both diagonals, $q$ is symmetric with respect to the lines $q_1q_3$ and $q_2q_4$, so that $A_1=A_3$ and $A_2=A_4$. By Lemma~\ref{lem:areas}(a), the four numbers $|A_l|$ are then equal. By \eqref{eq:dziobek}, $s_{ij}=\lambda'+\sigma A_iA_j/(m_im_j)$ takes one value on the four sides and another on the two diagonals, so $q$ is a rhombus with equal diagonals, that is, a square. A counterclockwise square with vertices $q_1,q_2,q_3,q_4$ in this order and centre $o$ satisfies $q_k-o=J^{k-1}(q_1-o)$, so it equals $(q_1-o)(1,J,-1,-J)+o\mathbf 1$, and its class is the class of $(1,J,-1,-J)$, which lies in $\cS^+$. Hence $N(\frac14,\frac14,\frac14,\frac14)=1$, and so $N\equiv1$. By Lemma~\ref{lem:local}, $\mathrm{pr}$ is then a bijective local real-analytic diffeomorphism, hence a real-analytic diffeomorphism. Its inverse has the form $m\mapsto(s(m),m)$, with $s:\mathcal M\to\cS^+$ real-analytic.

Now let $m\in(0,\infty)^4$, put $\hat m=m/\sum_im_i$, and let $q$ be a convex CC for $m$ with cyclic order $(1234)$. It is also a CC for $\hat m$. If $q$ is counterclockwise, then $[q]$ corresponds to a point of $\mathrm{pr}^{-1}(\hat m)$, so $[q]=s(\hat m)$. If $q$ is clockwise, then $\bar q$ is a counterclockwise convex CC with the same cyclic order, and $[q]=\iota(s(\hat m))$. The classes $s(\hat m)$ and $\iota(s(\hat m))$ are distinct, because the signs of the $A_l$ differ, and they are mirror images of each other. Consequently there are exactly two convex CCs with cyclic order $(1234)$ up to orientation-preserving similarity, and exactly one up to similarity. Both are nondegenerate local minima of $f_m$ by Corollary~\ref{cor:nondeg}, and both depend real-analytically on $m$. For the cyclic orders $(1243)$ and $(1324)$, exchange the labels $3$ and $4$, respectively $2$ and $3$, together with the corresponding masses.
\end{proof}

\begin{remark}\label{rem:existence}
Corollary~\ref{cor:nondeg} and the covering argument also prove the theorem of MacMillan and Bartky \cite{MB32} that convex CCs with every cyclic order exist for all masses; compare \cite{Xia04}. Indeed, the inequality $N\ge1$ uses only the elementary fact that the square is a CC for four equal masses.
\end{remark}

\section{The optimal constant}\label{sec:optimal}

We show that $\lmax(G)$ comes arbitrarily close to $\frac13$ at convex CCs.

\begin{figure}[ht]
\centering
\begin{tikzpicture}[scale=3.2]
\coordinate (O) at (0,0);
\coordinate (q1) at (0:1);
\coordinate (q2) at (67.1:1);
\coordinate (q3) at (180:0.15);
\coordinate (q4) at (247.1:0.15);
\draw[thick] (q1) -- (q2) -- (q3) -- (q4) -- cycle;
\draw[dashed] (q1) -- (q3) (q2) -- (q4);
\fill (O) circle[radius=0.012];
\node at (123.5:0.075) {$O$};
\node[right] at (q1) {$q_1$};
\node[above] at (q2) {$q_2$};
\node[left] at (q3) {$q_3$};
\node[below] at (q4) {$q_4$};
\draw (0:0.3) arc[start angle=0,end angle=67.1,radius=0.3];
\node at (33.5:0.42) {$\gamma_b$};
\end{tikzpicture}
\caption{The trapezoid $q_b$ of Lemma~\ref{lem:optimal} (sketch): $|q_1|=|q_2|=1$, $q_3=-b\,q_1$, $q_4=-b\,q_2$. As $b\to0$, the bodies $q_3$ and $q_4$ collide at the origin $O$, and $\gamma_b\to\frac\pi3$.}\label{fig:trapezoid}
\end{figure}

\begin{lemma}\label{lem:optimal}
For $b>0$ and $0<\gamma<\pi$ let $q(b,\gamma)$ be the configuration
\[
q_1=(1,0),\qquad q_2=(\cos\gamma,\sin\gamma),\qquad q_3=-b\,q_1,\qquad q_4=-b\,q_2 ,
\]
an isosceles trapezoid with the parallel sides $q_1q_2$ and $q_3q_4$ (Figure~\ref{fig:trapezoid}). For all sufficiently small $b>0$ there is $\gamma_b$ such that $q_b=q(b,\gamma_b)$ is a convex CC for positive masses with $m_1=m_2$ and $m_3=m_4$. As $b\to0$,
\[
\cos\gamma_b=\tfrac12-\tfrac34\,b+O(b^2),\qquad \frac{m_3}{m_1}\sim\tfrac34\,b^3,\qquad G(q_b)\to G_0=\begin{pmatrix}\frac5{16}&\frac{\sqrt3}{48}\\[2pt]\frac{\sqrt3}{48}&\frac{13}{48}\end{pmatrix},
\]
and $G_0$ has the eigenvalues $\frac14$ and $\frac13$. Moreover $\mu_1(q_b)\to\frac13$, and the smallest generalized eigenvalue of $(Q,K)$ at $q_b$ tends to $\frac23$. In particular, no constant larger than $\frac23$ can replace $\frac23$ in the bound $Q\ge\frac23K$.
\end{lemma}

So the two small bodies $q_3,q_4$ collide at the origin, which forms an equilateral triangle with $q_1$ and $q_2$ in the limit.

\begin{proof}
For $b>0$ and $0<\gamma<\pi$ the diagonals $q_1q_3$ and $q_2q_4$ meet at the origin, which is an interior point of both. Hence $q(b,\gamma)$ is a strictly convex quadrilateral with its vertices in counterclockwise order, and $q_3-q_4=-b(q_1-q_2)$, so $q_1q_2$ and $q_3q_4$ are parallel. The reflection in the bisector of the angle between $q_1$ and $q_2$ exchanges $q_1$ with $q_2$ and $q_3$ with $q_4$, so $r_{13}=r_{24}=1+b$ and $r_{14}=r_{23}$. Moreover $R_{12}=2-2\cos\gamma$, $R_{14}=1+2b\cos\gamma+b^2$ and $r_{34}=b\,r_{12}$, and the oriented areas are
\[
A_1=-A_2=\tfrac12\,b(1+b)\sin\gamma,\qquad A_3=-A_4=\tfrac12(1+b)\sin\gamma .
\]

\emph{Step 1: one equation.} Put $s_e=r_e^{-3}$, $\lambda'=\frac12(s_{13}+s_{14})$ and $w_e=s_e-\lambda'$. Then $w_{13}=w_{24}=\frac12(s_{13}-s_{14})$ and $w_{14}=w_{23}=-w_{13}$. Suppose that
\begin{equation}\label{eq:opt-shape}
w_{12}\,w_{34}=w_{13}^2,\qquad w_{13}<0,\qquad w_{34}>0 .
\end{equation}
Then $w_{12}>0$, and the masses $m_1=m_2=A_1w_{12}^{-1/2}$ and $m_3=m_4=A_3w_{34}^{-1/2}$ are positive. They satisfy Dziobek's relations $m_im_jw_{ij}=-A_iA_j$ of Lemma~\ref{lem:dziobek}, with $\sigma=-1$, for every edge. For the edges $12$ and $34$ this is the definition of the masses. For the other four edges it follows from $(w_{12}w_{34})^{1/2}=-w_{13}$, $A_2A_4=A_1A_3$ and $A_1A_4=A_2A_3=-A_1A_3$. Hence $\omega_e=m_im_jw_e=-A_iA_j$, and for each $i$
\[
\sum_{j\ne i}\omega_{ij}\,(q_j-q_i)=-A_i\sum_{j\ne i}A_j\,(q_j-q_i)=0
\]
by Lemma~\ref{lem:areas}(a). The configuration $q(b,\gamma)$ therefore satisfies \eqref{eq:equil}, and hence it is a CC for these masses, with $\lambda=M\lambda'$ (see the converse stated after \eqref{eq:equil}).

\emph{Step 2: the branch.} Since $s_{34}=b^{-3}r_{12}^{-3}$, the equation in \eqref{eq:opt-shape} is equivalent, for $b>0$, to
\[
(s_{12}-\lambda')(1-b^3r_{12}^3\lambda')=\tfrac14\,b^3r_{12}^3(s_{13}-s_{14})^2 .
\]
The difference of the two sides is a real-analytic function of $(\gamma,b)$ near $(\frac\pi3,0)$. At $b=0$ we have $s_{13}=s_{14}=\lambda'=1$, and the difference is $(2-2\cos\gamma)^{-3/2}-1$. It vanishes at $\gamma=\frac\pi3$, and its derivative there is $-3\sin\frac\pi3\ne0$. By the implicit function theorem there is a real-analytic function $b\mapsto\gamma_b$ with $\gamma_0=\frac\pi3$ that solves the equation for all small $b$. Expanding in $b$ gives $s_{13}=1-3b+O(b^2)$ and $s_{14}=1-3b\cos\gamma_b+O(b^2)$, hence
\[
\lambda'=1-\tfrac32(1+\cos\gamma_b)\,b+O(b^2),\qquad s_{14}-s_{13}=3(1-\cos\gamma_b)\,b+O(b^2)=\tfrac32\,b+O(b^2).
\]
The right side of the equation is $O(b^5)$, so $s_{12}=\lambda'+O(b^5)$. Since $s_{12}=(2-2\cos\gamma_b)^{-3/2}=1+3(\cos\gamma_b-\frac12)+O(b^2)$, this gives $\cos\gamma_b=\frac12-\frac34b+O(b^2)$. Finally $w_{13}=-\frac34b+O(b^2)<0$ and $b^3w_{34}=r_{12}^{-3}-b^3\lambda'\to1$. So \eqref{eq:opt-shape} holds for all small $b>0$, and by Step~1, $q_b=q(b,\gamma_b)$ is a convex CC for positive masses with $m_1=m_2$ and $m_3=m_4$.

\emph{Step 3: the masses.} By Step~1, $m_3/m_1=(A_3/A_1)(w_{12}/w_{34})^{1/2}=b^{-1}|w_{13}|/w_{34}$. Here $|w_{13}|=\frac34b+O(b^2)$ and $w_{34}=b^{-3}(1+O(b))$, so $m_3/m_1=\frac34b^3(1+O(b))$.

\emph{Step 4: the limit of $G$.} We use Proposition~\ref{prop:midG}, with $d_e=-w_eA_iA_jr_e^5/3$ and $n_e=q_i+q_j$. As $b\to0$ along the branch, $q_2\to q_2^0=(\frac12,\frac{\sqrt3}2)$, $r_e\to1$ for $e\ne34$, $r_{34}/b\to1$, $A_1/b\to\frac{\sqrt3}4$ and $A_3\to\frac{\sqrt3}4$. Moreover $w_{13}/b=w_{24}/b\to-\frac34$, $w_{14}/b=w_{23}/b\to\frac34$, $b^3w_{34}\to1$ and $w_{12}=w_{13}^2/w_{34}=O(b^5)$. Hence
\[
\frac{d_{13}}{b^2},\ \frac{d_{24}}{b^2},\ \frac{d_{14}}{b^2},\ \frac{d_{23}}{b^2}\to\frac3{64},\qquad \frac{d_{34}}{b^2}\to\frac1{16},\qquad \frac{d_{12}}{b^2}\to0 ,
\]
while $n_{13},n_{14}\to q_1$, $n_{24},n_{23}\to q_2^0$ and $n_{34}\to0$; therefore $W/b^2\to\frac14$, $\bar n\to\frac38(q_1+q_2^0)$ and
\[
\frac C{b^2}\to\frac3{32}\bigl(q_1q_1^{\mathsf T}+q_2^0(q_2^0)^{\mathsf T}\bigr)-\frac9{256}(q_1+q_2^0)(q_1+q_2^0)^{\mathsf T}=\frac1{1024}\begin{pmatrix}39&-3\sqrt3\\-3\sqrt3&45\end{pmatrix}.
\]
Since $q_3=-bq_1$, $q_4=-bq_2$, $A_2=-A_1$ and $A_4=-A_3$, we have $M_2=(A_1+b^2A_3)(q_1q_1^{\mathsf T}-q_2q_2^{\mathsf T})$. Thus
\[
\frac{M_2}b\to\frac{\sqrt3}4\bigl(q_1q_1^{\mathsf T}-q_2^0(q_2^0)^{\mathsf T}\bigr)=\frac{\sqrt3}{16}\begin{pmatrix}3&-\sqrt3\\-\sqrt3&-3\end{pmatrix},
\]
and $G=(M_2/b)^{-1}(C/b^2)(M_2/b)^{-1}$ converges to the matrix $G_0$ of the statement. Its trace is $\frac7{12}$ and its determinant is $\frac1{12}$, so its eigenvalues are $\frac14$ and $\frac13$.

\emph{Step 5: optimality.} The eigenvalues depend continuously on the matrix, so $\mu_1(q_b)=\lmax(G(q_b))\to\frac13$. By Proposition~\ref{prop:majorant}(c), the smallest generalized eigenvalue of $(Q,K)$ at $q_b$ is $1-\mu_1(q_b)$, and it tends to $\frac23$. For a corresponding eigenvector $v$ we have $K(v)>0$ and $Q(v)=(1-\mu_1(q_b))K(v)$. This proves the last assertion of the lemma.
\end{proof}

\begin{remark}\label{rem:second-order}
In fact $\mu_1(q_b)=\frac13-\frac14b^2+O(b^3)$, so the first-order term vanishes. The matrix $q_1q_1^{\mathsf T}-q_2q_2^{\mathsf T}$ is $\sin\gamma$ times a symmetric matrix whose square is the identity, and $(A_1+b^2A_3)\sin\gamma=2A_1A_3$. So $G=M_2^{-1}CM_2^{-1}$ has the eigenvalues of $C/(4A_1^2A_3^2)$. The reflection in the bisector of the angle between $q_1$ and $q_2$, the axis of the trapezoid, fixes $n_{12}$ and $n_{34}$, exchanges $n_{13}$ with $n_{24}$ and $n_{14}$ with $n_{23}$, and preserves the weights, since $d_{13}=d_{24}$ and $d_{14}=d_{23}$. Hence $\bar n$ lies on the axis, and the unit vector $(\sin\frac\gamma2,-\cos\frac\gamma2)$ normal to the axis is an eigenvector of $C$. The components of $n_{13},n_{24},n_{14},n_{23}$ along it are $\pm(1-b)\sin\frac\gamma2$ and $\pm(1+b)\sin\frac\gamma2$, and those of $n_{12}$ and $n_{34}$ vanish. With $d_{13}=\frac13|w_{13}|A_1A_3(1+b)^5$, $d_{14}=\frac13|w_{13}|A_1A_3r_{14}^5$, $|w_{13}|=\frac12(s_{14}-s_{13})$ and $\sin^2\gamma=2\sin^2\frac\gamma2\,(1+\cos\gamma)$, the corresponding eigenvalue of $G$ is
\[
\mu_\perp=\frac{(s_{14}-s_{13})\bigl((1+b)^3(1-b)^2+r_{14}^5\bigr)}{6b\,(1+\cos\gamma_b)} .
\]
Its limit is $\frac32\cdot\frac29=\frac13$, and the other eigenvalue $\tr G-\mu_\perp$ tends to $\frac7{12}-\frac13=\frac14$ by Step~4, so $\mu_\perp=\mu_1(q_b)$ for small $b$. Taking the equation $s_{12}=\lambda'+O(b^5)$ of Step~2 to second order, with $\lambda'=1-\frac94b+\frac{69}{16}b^2+O(b^3)$ and $s_{12}=1+3(\cos\gamma_b-\frac12)+\frac{15}2(\cos\gamma_b-\frac12)^2+O(b^3)$, gives $\cos\gamma_b=\frac12-\frac34b+\frac1{32}b^2+O(b^3)$. Then $s_{14}-s_{13}=\frac32b-\frac{27}8b^2+\frac{187}{32}b^3+O(b^4)$, $r_{14}^5=1+\frac52b+\frac58b^2+O(b^3)$, $(1+b)^3(1-b)^2=1+b-2b^2+O(b^3)$ and $1+\cos\gamma_b=\frac32-\frac34b+\frac1{32}b^2+O(b^3)$. Hence
\[
\mu_\perp=\frac13\cdot\frac{1-\frac12b-\frac{35}{48}b^2}{1-\frac12b+\frac1{48}b^2}+O(b^3)=\frac13-\frac14b^2+O(b^3).
\]
\end{remark}

\section{Consequences and open questions}\label{sec:cons}

\subsection{Proof of Corollary~B}
(i) By the proof of Theorem~A, the counterclockwise and the clockwise convex CC with cyclic order $(1234)$ are the classes $s(\hat m)$ and $\iota(s(\hat m))$, where $\hat m=m/\sum_im_i$. Both depend real-analytically on $m$. For the other cyclic orders, relabel the bodies.

(ii) A similarity preserves the cyclic order of a convex configuration. By Theorem~A there are therefore exactly three convex CCs up to similarity, one for each cyclic order. Each of them gives two classes up to orientation-preserving similarity, the configuration and its mirror image, which are distinct by Lemma~\ref{lem:areas}(b). This gives six. All of them are nondegenerate local minima by Corollary~\ref{cor:nondeg}.

Now suppose that all concave CCs for the masses $m$ are nondegenerate. Collinear CCs are nondegenerate, of Morse index $2$ \cite[Proposition~19]{Moe15}, and convex CCs are nondegenerate by Corollary~\ref{cor:nondeg}. Hence $f_m$ is a Morse function on $\cS\setminus\Delta$. It is positive and tends to $+\infty$ at $\Delta$, so its sublevel sets are compact. Its critical set is compact by Lemma~\ref{lem:shub}, applied with constant masses, and it consists of isolated points, so it is finite. The Morse inequalities \cite[Section~11]{Moe15}, \cite{Pac87} compare the numbers $\gamma_k$ of critical points of index $k$ with the Betti numbers of $\cS\setminus\Delta$, whose Poincaré polynomial is $1+5t+6t^2$ \cite[Proposition~24]{Moe15}. By Corollary~\ref{cor:palmore} all indices are at most $2$, so
\[
\gamma_0+\gamma_1t+\gamma_2t^2=1+5t+6t^2+(1+t)(r_0+r_1t)
\]
with $r_0,r_1\ge0$. Hence $\gamma_0=1+r_0$, $\gamma_1=5+r_0+r_1$, $\gamma_2=6+r_1$, and the total number of CCs is $12+2(r_0+r_1)$. The six convex CCs give $\gamma_0\ge6$, and the twelve collinear CCs \cite[Propositions~18 and~19]{Moe15} give $\gamma_2\ge12$. Thus $r_0\ge5$ and $r_1\ge6$, and there are at least $34$ CCs up to orientation-preserving similarity. This is the argument of \cite[Section~13]{Moe15}, where all CCs are assumed to be nondegenerate. By Corollary~\ref{cor:nondeg} the assumption is needed only for the concave CCs.

(iii) By Lemmas~\ref{lem:nocollinear} and~\ref{lem:areas}(d), a planar CC is collinear, convex or concave. Collinear CCs are nondegenerate \cite[Proposition~19]{Moe15}, \cite{Pac87}, and convex CCs are nondegenerate by Corollary~\ref{cor:nondeg}. \qed

\subsection{Proof of Corollary~C}
The hypotheses and the conclusions are invariant under reflections, so we may assume that $q$ is counterclockwise. By the proof of Theorem~A, two counterclockwise convex CCs with cyclic order $(1234)$ for the same masses differ by an orientation-preserving similarity $T(z)=\mu z+t$. If such a $T$ fixes two distinct points, then $T$ is the identity.

(a) Let $\mathfrak r$ be the reflection in the line $q_2q_4$. Then $\mathfrak rq$ is a clockwise convex CC for the masses $m$. Exchanging the labels $1$ and $3$, we see that $q'=(\mathfrak rq_3,q_2,\mathfrak rq_1,q_4)$ is a CC for the masses $(m_3,m_2,m_1,m_4)=m$. It traverses the vertices of $\mathfrak rq$ in the order $3,2,1,4$, the reverse of their order in $\mathfrak rq$. Hence $q'$ is a counterclockwise convex CC with cyclic order $(1234)$, and $q'_i=T(q_i)$ for an orientation-preserving similarity $T$. Since $T$ fixes $q_2$ and $q_4$, it is the identity, and $\mathfrak rq_3=q_1$, that is, $q$ is symmetric with respect to the diagonal $q_2q_4$.

(b) Let $\mathfrak r$ be the reflection in the perpendicular bisector of $q_1q_2$, so that $\mathfrak rq_1=q_2$ and $\mathfrak rq_2=q_1$. Exchanging the labels $1\leftrightarrow2$ and $3\leftrightarrow4$ in $\mathfrak rq$ gives the CC $q'=(\mathfrak rq_2,\mathfrak rq_1,\mathfrak rq_4,\mathfrak rq_3)=(q_1,q_2,\mathfrak rq_4,\mathfrak rq_3)$ for the masses $(m_2,m_1,m_4,m_3)=m$. It traverses the vertices of $\mathfrak rq$ in the order $2,1,4,3$, which is again the reverse order. As in (a), $q'_i=T(q_i)$, where now $T$ fixes $q_1$ and $q_2$. Therefore $T$ is the identity, and $\mathfrak rq_4=q_3$, that is, $\mathfrak r$ exchanges $q_3$ and $q_4$. The side $q_3q_4$ is then perpendicular to the axis of $\mathfrak r$, like $q_1q_2$, and $r_{14}=r_{23}$, so $q$ is an isosceles trapezoid with $q_1q_2\parallel q_3q_4$.

(c) By (a), $q$ is symmetric with respect to the line $q_2q_4$. The configuration $(q_2,q_3,q_4,q_1)$, with the masses $(m_2,m_3,m_4,m_1)$, is a convex CC with cyclic order $(1234)$ whose first and third masses are equal. By (a) it is symmetric with respect to the line $q_3q_1$. Each diagonal therefore lies on the perpendicular bisector of the other, so the diagonals bisect each other and are perpendicular, and $q$ is a rhombus. \qed

\subsection{Open questions}
\begin{enumerate}[label=\textup{(\arabic*)},leftmargin=2em]
\item \emph{A computer-free proof of the sharp bound.} It would be interesting to prove $\lmax(G)<\frac13$ by hand. Along the trapezoids of Section~\ref{sec:optimal} such a proof has to be sharp to second order (Remark~\ref{rem:second-order}).
\item \emph{Dziobek configurations.} By Remark~\ref{rem:identity}(ii), $Q=K+\sigma|\sum_lA_lv_l|^2$ for Dziobek configurations of $d+2$ bodies in $\R^d$. Is there an analogue of Corollary~\ref{cor:nondeg} for some class of Dziobek CCs?
\item \emph{Other homogeneous potentials.} For the potential $\sum_{i<j}m_im_jr_{ij}^{-a}$ with $a>0$, the same computation gives $\kappa_e=a(a+2)m_im_jr_e^{-a-2}$ and $w_e=a\,r_e^{-a-2}-\lambda'$. The proof of Lemma~\ref{lem:dziobek} again gives $\sigma<0$, and $D_e=|\sigma|/\kappa_e$ is again independent of the masses, so Propositions~\ref{prop:identity} and~\ref{prop:majorant} carry over. The representations of $L$ in Lemma~\ref{lem:midpoint-family} depend only on the configuration, so $G$ again has the form $M_2^{-1}\widetilde SM_2^{-1}$, with weights adapted to the potential. An analogue of Theorem~A would need an analogue of the bound $\lmax(G)<1$ (Corollary~\ref{cor:nondeg}). Given that, Lemma~\ref{lem:collinear} and the base case with equal masses carry over: the computation in the proof of Lemma~\ref{lem:collinear} uses only that $r^{-a-2}$ is decreasing, and \cite[Theorem~1]{AFS08} holds for these potentials. We have not attempted this.
\item \emph{Concave CCs.} Proposition~\ref{prop:majorant}(c) applies to concave CCs as well: a concave CC is degenerate if and only if $1$ is an eigenvalue of its matrix $G$. Degenerate concave CCs exist \cite{Rob25,LX25,LPRX26}. It would be interesting to describe the set of masses for which they occur.
\end{enumerate}

\appendix

\section{The numerical facts (N1)--(N15)}\label{app:handN}\label{sec:hN-num}
This appendix collects the numerical facts (N1)--(N15) used in Section~\ref{sec:handN}, in the notation of that section. Every fact can be checked by hand from the displayed rational numbers, with exact arithmetic and, where an item says so, an elementary monotonicity or concavity argument. Each of the constants $0.0166$, $0.035$, $0.034$, $0.905$ and $0.92$ of Section~\ref{sec:handN} is the bound given by its estimate, rounded in the safe direction to the digits shown.
A square or cube root is first enclosed by rationals, and the enclosure is checked by squaring or cubing. The first two items enclose the constants $V_0(1)=\frac{4-\sqrt2}3$ and $\psi_*=4(49/256)^{1/3}-1$.
\begin{enumerate}[label=\textup{(N\arabic*)},leftmargin=3.2em]
\item $1.4142135^2<2<1.4142136^2$, so $0.8619288<V_0(1)<0.8619289$. These enclosures are used in the proof of Lemma~\ref{lem:hN-const} and in (N4), (N5), (N9)--(N12) and (N15).
\item $0.5763^3<\frac{49}{256}<0.57631^3$, so $1.3052<\psi_*<1.30524$. This enclosure is used in (N3), (N5), (N9), (N12) and (N15).
\item $(4\cdot0.6644^2-1)\cdot1.30524<1$ and $2.6644<0.4811\cdot2\cdot1.6644^2$. So $\sqrt{(1+1/\psi_*)/4}>0.6644$, and since $g'(x)=(2+\sqrt x)/(2(1+\sqrt x)^2)$ decreases in $\sqrt x$, $g'(\frac{1+1/\psi_*}4)<0.4811$. Used in the proof of Lemma~\ref{lem:hN-two}(b) and in (N4) and (N12).
\item $\frac1{2\cdot1.4142136}-\frac{0.4811}{12}\ge0.31346$ and $(0.0166-\frac{0.4811}3)^2<4\cdot0.31346\cdot0.0166$. With (N3), this gives the bound $0.0166$ in Lemma~\ref{lem:hN-two}(b).
\item $0.37157^2<1-0.8619289$ and $1-0.8619288<0.37158^2$, so $0.37157<z_\star<0.37158$. Further $1.190812^2\,(1+8\cdot0.37157^3)\ge2$, $0.9284^2\le1-0.37158^2$, and
\[
2\cdot1.30524\cdot1.190812\cdot\frac{1+2\cdot0.37158+4\cdot0.37158^2}{12\cdot0.37157^2\cdot0.9284}\le4.64 .
\]
Used in the proof of Lemma~\ref{lem:hN-two}(c) and in (N7) and (N8).
\item For $0\le z\le\frac12$ we have $(1+z)^2\ge1$, $1-z\ge\frac12$, $z^2\le\frac14$ and $1+2z+4z^2\le3$, so $2(1+z)^2(1-z)\ge1>\frac34\ge z^2(1+2z+4z^2)$. Since $(\log H)'(z)\cdot z(1-z^2)(1+2z+4z^2)=z^2(1+2z+4z^2)-2(1+z)^2(1-z)$, the function $H$ is decreasing on $(0,\frac12]$. Used in the proof of Lemma~\ref{lem:hN-two}(c).
\item The quadratic polynomial
\[
0.035\bigl(1-z^2-4.64(\tfrac12-z)\bigr)-4.64\cdot0.87158\,(\tfrac12-z)(z-0.37157)
\]
has the leading coefficient $4.64\cdot0.87158-0.035>0$, the coefficient $0.035\cdot4.64-4.64\cdot0.87158\cdot0.87157$ of $z$ and the constant term $0.035\,(1-2.32)+4.64\cdot0.87158\cdot0.37157/2$, and its discriminant is negative:
\begin{multline*}
\bigl(0.035\cdot4.64-4.64\cdot0.87158\cdot0.87157\bigr)^2\\
<4\,\bigl(4.64\cdot0.87158-0.035\bigr)\bigl(0.035\,(1-2.32)+4.64\cdot0.87158\cdot0.37157/2\bigr).
\end{multline*}
Hence it is positive for all $z$; with (N5) this gives the last inequality in the proof of Lemma~\ref{lem:hN-two}(c).
\item $1-0.37158^2-4.64\,(\frac12-0.37157)\ge0.266$. Used in the proof of Lemma~\ref{lem:hN-two}(c).
\item $1.14249^2\ge1.30524$, $1.7320508^2\le3$, $0.37118^2\le0.835\cdot0.165$, and
\[
(1.30524\cdot1.14249-1)\Bigl(2\cdot0.8619289+0.0166-0.8-\tfrac43\bigl(0.3325+0.8660254\cdot0.37118\bigr)\Bigr)\le0.034 ;
\]
$0.542885^3\ge\frac4{25}$, $1.082378^2\ge4\cdot0.542885-1$, $0.30405^2\le0.8969289\cdot0.1030711$, and
\begin{multline*}
\bigl((4\cdot0.542885-1)\cdot1.082378-1\bigr)\\
\times\Bigl(0.8619289+0.0166-\tfrac43\bigl(0.75-\tfrac{0.8969289}2+0.8660254\cdot0.30405\bigr)\Bigr)\le0.034 .
\end{multline*}
These bound the two products in the proof of Lemma~\ref{lem:hN-const}, because $X(1-X)$ is decreasing for $X>\frac12$. Its bounds $1.7320508^2\le3$ and $0.30405^2\le0.8969289\cdot0.1030711$ are also used in (N10), and the first in (N12).
\item $0.41082^2\le0.785\cdot0.215$ and $\frac43\min\bigl(0.8660254\cdot0.57+0.41082,\ 0.8660254\cdot(2\cdot0.8969288-1)+0.30405\bigr)>1$. Here $b'(x)=-1+\frac43\sin(240^\circ-2\vartheta_X)\,\frac{d\vartheta_X}{dx}$ with $\frac{d\vartheta_X}{dx}=1/(2\sqrt{X(1-X)})\ge1$ and $\sin(240^\circ-2\vartheta_X)=\frac{\sqrt3}2(2X-1)+\sqrt{X(1-X)}$. The angle $240^\circ-2\vartheta_X$ lies in $(60^\circ,120^\circ)$, where the sine is concave, so its minimum over $x\in[\frac34,V_0(1)]$ is attained at an end; the two numbers bound it there. Hence $b'>0$. Used in the proof of Lemma~\ref{lem:hN-const}.
\item $0.8619289+0.0083+0.0345\le0.905$. This gives the bound $\bar V\le0.905$ in Lemma~\ref{lem:hN-const}.
\item $0.526632^2\le3\cdot0.8969289\cdot0.1030711$ and $0.8969289-0.5+0.526632\ge0.92$. Further $V^*(\psi_*)\le\frac43(1-0.6644^3)+\frac{0.4811}{12}(1-\frac1{1.30524})(5-\frac1{1.30524})\le0.98201$ by (N2) and (N3), $0.13412^2\le1-0.98201$, $0.99096^2\le0.98201$, and $\frac{0.13412}2+0.8660254\cdot0.99096\ge0.92$. Used in the proofs of Lemmas~\ref{lem:hN-angles} and~\ref{lem:hN-const}.
\item $F''(\frac12)=\frac{48}4-\frac{1.92}2-10.2688=0.7712>0$. Used in the proof of Lemma~\ref{lem:hN-quartic}.
\item $T(\frac12)=\frac{2730393}{25000000}-0.29\cdot\frac{273}{31250}>0$ and $T(k_*)=\frac{2730393}{25000000}+\bigl(\frac{0.81}{0.92}-0.79\bigr)\frac{273}{31250}>0$. Used in the proof of Lemma~\ref{lem:hN-quartic}.
\item $(1.30524-1)^2\le0.14996\cdot4\cdot1.30524\cdot0.8619289\cdot(1-0.8619289)$, $0.85596^2\le0.8619288\cdot0.85004$, $0.14389^2\le0.1380711\cdot0.14996$, $0.94707^2\ge0.8969289$, and $0.85596+0.14389>0.94707$. Here $V(1-V)$ is decreasing for $V>\frac12$. Used in the proof of Lemma~\ref{lem:hN-angles}.
\end{enumerate}

\section*{Data availability}
This paper has no associated code or data. The code and the formal proofs of the companion paper are described in \cite{TomB}.

\section*{Acknowledgements}
Generative AI tools were used to write code, run analyses and draft text; the author directed the work and takes full responsibility for its content.

\end{document}